\documentclass[reqno]{amsart}
\usepackage{fullpage}
\usepackage{array}
\usepackage{mathrsfs}
\usepackage{hyperref}
\usepackage{graphicx,color}
\usepackage{float}

\newcommand{\varied}

\newcommand{\Fal}[1]{\ensuremath{F\langle{#1}\rangle}}

\newcommand{\gi}{\ensuremath{(G,\ast)}}
\newcommand{\Idd}{\ensuremath{\textnormal{Id}}^{\sharp}}

\newcommand{\varGstar}{\ensuremath{\textnormal{var}^{\sharp}}}

\newcommand{\spam}{\ensuremath{\textnormal{span}}}

\usepackage{longtable}

\newcommand{\black}{\color{black}}

\newcommand{\Gstar}{\ensuremath{(G,\ast)}-}
\newcommand{\supp}{\ensuremath{\mbox{supp}}}

\newcommand{\IdGstar}[1]{\ensuremath{\textnormal{Id}^{\sharp}({#1})}}

\newtheorem{theorem}{Theorem}[section]
\newtheorem{example}[theorem]{Example}

\newtheorem{lemma}[theorem]{Lemma}
\newtheorem{corollary}[theorem]{Corollary}
\newtheorem{remark}[theorem]{Remark}
\newtheorem{proposition}[theorem]{Proposition}
\newtheorem{definition}[theorem]{Definition}
\newcolumntype{C}[1]{>{\centering\let\newline\\\arraybackslash\hspace{0pt}}m{#1}}

\usepackage[utf8]{inputenc}
\usepackage{amsfonts}
\usepackage{amssymb}
\usepackage{verbatim}
\usepackage{booktabs}

\numberwithin{equation}{section}

\begin{document}
\title{A structural classification of algebras \\ with graded involution and quadratic codimension growth}

\author{Wesley Quaresma Cota$^{a,*}$, Luiz Henrique de Souza Matos$^{b}$ and Ana Cristina Vieira$^{b}$}
\thanks{{\it E-mail addresses:} quaresmawesley@gmail.com (Cota), henrique.lui.00@gmail.com (Matos), anacris@ufmg.br (Vieira).}




\thanks{\footnotesize $^{*}$ Corresponding author.}

\subjclass[2020]{Primary 16R10, 16W50, Secondary 20C30}

\keywords{Polynomial identities, graded involution, codimension growth}

    \dedicatory{$^{a}$ IME, USP, Rua do Matão 1010, 05508-090, São Paulo, Brazil \\ $^b$ ICEx, UFMG, Avenida Antonio Carlos 6627, 31123-970, Belo Horizonte, Brazil}

\begin{abstract}  
The theory of algebras with polynomial identities has developed significantly, with special attention devoted to the classification of varieties according to the asymptotic behavior of their codimension sequences. This sequence is a fundamental numerical invariant, capturing the growth rate of the polynomial identities of a given algebra. Several partial classification results have been obtained, with particular attention devoted to algebras equipped with additional structure. In this paper, we consider associative $G$-graded algebras endowed with a graded involution. We provide a complete classification, up to equivalence, of unitary algebras with quadratic codimension growth. Our approach establishes a direct correspondence between the algebras generating minimal varieties and the nonzero multiplicities appearing in the decomposition of the proper cocharacters.  As a consequence, we establish that every variety with at most quadratic growth is generated by an algebra that decomposes as a direct sum of algebras generating minimal varieties of at most quadratic growth.
\end{abstract}

\maketitle

\section{Introduction}

One of the central problems in the theory of algebras with polynomial identities is the classification of varieties according to the asymptotic behavior of their codimension sequences. This sequence was introduced by Regev as a fundamental invariant in PI-theory, capturing the dimension of the space of multilinear polynomials modulo the polynomial identities of the algebra. It is well known that the sequence of codimensions $\{c_n(A)\}_{n\in \mathbb{N}}$ effectively measures the growth of the polynomial identities of a given algebra. Indeed, over a field of characteristic zero, any polynomial identity follows from a finite set of multilinear ones (see \cite{Kemer}).

When an algebra $A$ satisfies a nontrivial polynomial identity, i.e.,  it is a PI-algebra, Regev \cite{RG} proved that its codimension sequence is exponentially bounded. A breakthrough in this direction was obtained by Kemer \cite{Kem}, who established that $c_n(A) \le \alpha n^t$, for some constants $\alpha$ and $ t$, if and only if both the Grassmann algebra $\mathcal{G}$ and the algebra $UT_2$ of $2 \times 2$ upper triangular matrices do not belong to the variety $\textnormal{var}(A)$ generated by $A$. A key consequence is that, over a field of characteristic zero, the codimension sequence of any PI-algebra is either polynomially bounded or grows exponentially.

Motivated by these important results, La Mattina \cite{LaMa} subsequently classified all subvarieties of the varieties generated by $\mathcal{G}$ and $UT_2$. The algebras arising in this classification have proven to be fundamental in several developments in PI-theory, and they continue to play a crucial role in the present work. Later, Giambruno and La Mattina \cite{GiamLaMa} completed the classification of varieties of linear codimension growth, showing that the linear case is fully understood. However, it has become increasingly evident that new tools are required to address the general classification of varieties of quadratic growth.

 Recall that $A$ generates a minimal variety of polynomial growth $n^t$ if $c_n(A)\approx qn^t$ while every proper subvariety has polynomial growth $n^p$ for some $p<t$. In this direction, Jorge and Vieira \cite{Sandra} classified the minimal varieties of quadratic growth. The study of minimal varieties is of particular importance since  their generating algebras form the basic building blocks in the construction of more complex varieties. 
 
 This perspective naturally leads to a central open question in PI-theory: is every variety of polynomial growth obtained, up to equivalence, as a direct sum of algebras generating minimal varieties? At present, the answer to this question remains unknown. 

Significant progress has been achieved in the unitary case, where the theory of multilinear proper polynomials offers a more refined framework for understanding the structure of these varieties. This approach was explored by Giambruno, Petrogradsky and La Mattina \cite{Petro}, and later by de Oliveira and Vieira \cite{Mara3}, who classified unitary varieties of polynomial growth satisfying $c_n(A)\leq \alpha n^4$.

In recent years, the theory of polynomial identities has broadened to include additional algebraic structures such as graded algebras, algebras with involution or superinvolution. This approach is justified by the observation that ordinary polynomial identities arise as special cases of identities defined with additional algebraic structures. Further progress toward the classification of varieties of polynomial growth equipped with additional structure has been achieved in \cite{Dafne, Wesley20, Cota2, Misso, Tatiana, Plamen,  Maralice}.

In this paper, we focus on $G$-graded algebras endowed with a graded involution. In this context, Cota and Vieira \cite{Wesley3} classified all minimal varieties of quadratic codimension growth, including the unitary ones. Moreover, the authors in \cite{Malu} obtained the classification of unitary superalgebras with graded involution of quadratic codimension growth. In \cite{Maralice}, the authors presented the  classification of the varieties generated by finite-dimensional $\gi$-algebras with at most linear codimension growth.

Here we address the complete classification, up to equivalence, of unitary varieties $\mathcal{V}$ of $G$-graded algebras with graded involution of quadratic codimension growth, where $G$ is a finite group and $\mathcal{V}$ is generated by a finite-dimensional algebra. We establish a direct connection between the algebras generating minimal varieties and the nonzero multiplicities appearing in the decomposition of the proper cocharacter of $A$. As a consequence, we classify all unitary varieties with quadratic codimension growth, and we prove that each of them is obtained as a direct sum of algebras generating minimal varieties. The results established here are not only of intrinsic interest, but also constitute a significant step toward a comprehensive classification of varieties of quadratic codimension growth.

\section{Unitary $(G,*)$-varieties}

Let $G=\{g_1,g_2,\dots,g_k\}$ be a finite multiplicative group with unit element $g_1=1$ and $A$ an associative algebra over a field $F$ of characteristic zero.  
We say that $A$ is a {$G$-graded algebra} if there exist subspaces $A_g$, $g \in G$, such that 
\[
A = \bigoplus_{g \in G} A_g 
\quad \text{and} \quad 
A_g A_h \subseteq A_{gh}, \quad \text{for all } g,h \in G.
\]

Each subspace $A_g$ is called the {homogeneous component of degree}~$g$. Moreover, the support $\mbox{supp}(A)$ of a $G$-graded algebra $A$ is defined as the set of elements $g\in G$ such that $A_g\neq \{0\}$.  

We may observe that any algebra $A$ can be regarded as a $G$-graded algebra with the trivial grading, defined by $A_1 = A$ and $A_g = \{0\}$ for all $g \in G^{\times}$, where $G^{\times}$ denotes the set $ G - \{1\}$.

\begin{example}  
     Let $M_n(F)$ be the algebra of matrices of order $n$ and $\mathsf{g}=(g_1, \ldots, g_n) \in G^n$ be an arbitrary $n$-tuple of elements of $G$. We may use $\mathsf{g}$ to define a $G$-grading on $M_n(F)$ as follows: $(M_n(F))_{g}= \textnormal{span}_F\{e_{ij}\mid g_i^{-1}g_j=g\}, g\in G.$ This $G$-grading is called elementary $G$-grading induced by $\mathsf{g}$ and will also be considered in certain subalgebras of $M_n(F)$.
\end{example}

    We recall that an involution on $A$ is a linear map satisfying $(a^{*})^{*}=a$ and $(ab)^{*}=b^*a^*$, for all $ a,b\in A$, and an algebra endowed with an involution is called a $*$-algebra.

\begin{example}
    On the algebra $UT_n$ of upper triangular matrices of size $n$, we can consider the reflection involution $\rho$ defined by $e_{ij}^{\rho}=e_{n-j+1,\, n-i+1},\, 1\leq i\leq j\leq n,$ where $e_{ij}$ denotes the elementary matrix with $1$ in the $(i,j)$-entry and $0$ elsewhere. 
\end{example}

  Here we are interested in $G$-graded algebras endowed with an involution that is compatible with the grading, in the following sense.  

  \begin{definition}
      An involution $*$ on a $G$-graded algebra $A$ is said to be {graded} if $A_g^{*} = A_g$, for all $g \in G.$ A $G$-graded algebra $A$ equipped with a graded involution $*$ is called a {$(G,*)$-algebra}.
  \end{definition}

 Observe that if $A_gA_h\neq \{0\}$, for some $g,h\in G$, then we must have $gh=hg$.
    
    For a $\gi$-algebra $A$ one can further decompose
	$$A=\underset{g\in G}{\bigoplus}(A_{g}^+ +  A_{g}^-)$$ where $A_{g}^+= \{a\in A_g \mid a^*=a\}$ and $A_{g}^-= \{a\in A_g \mid a^*=-a\}$  are the homogeneous symmetric and homogeneous skew components of degree $g$, respectively. In this case,  we define its $\gi$-support as $$\mbox{supp}^*(A)= \{g^\epsilon \mid A_{g}^\epsilon \neq \{0\},\, g\in G , \epsilon \in \{+,-\} \}.$$

An isomorphism of $(G,*)$-algebras is an isomorphism $\varphi\colon A \to B$ that preserves both the grading and the involution, that is, $\varphi(A_g)=B_g$ for all $g\in G$ and $\varphi(a^*)=\varphi(a)^*$ for all $a\in A$. When such a map exists, we say that $A$ and $B$ are isomorphic as $(G,*)$-algebras.

    For all $g\in G$, consider $ X_g^{*}= \{x_{i,g}, x_{i,g}^{*}\mid g\in G, i\geq 1\}$ a countable set of variables and define $X= \bigcup_{g\in G} X_g^{*}$. Let $F\langle X\mid G, * \rangle$ be the free associative \Gstar algebra generated by $X$ over $F$, whose elements are called \Gstar polynomials. In the following, it will be convenient to consider symmetric and skew homogeneous variables as follows: consider $X^+={\bigcup_{g\in G} }X_g^+$ and $X^-={\bigcup_{g\in G} }X_g^-$,  where   $X_g^+= \{x_{i,g}^+:= x_{i,g}+x_{i,g}^*\mid  i\geq 1\}$ is the set of homogeneous symmetric variables of degree $g$ and  $X_g^-= \{x_{i,g}^-=x_{i,g}-x_{i,g}^*\mid  i\geq 1\}$ is the set of homogeneous skew variables of degree $g$. Then, we have $\mathcal{F}:=\Fal{X\mid G,*}= \Fal{X^+\cup X^-}$.
    For all $g\in G$, consider $\mathcal{F}_g$ the space of elements that have homogeneous degree $g$ and notice that $\mathcal{F}= \bigoplus_{g\in G} \mathcal{F}_{g}$ has a structure of \Gstar algebra.

    We use the notation ${x_{i,g}^\epsilon}$ to indicate a variable that can be either symmetric or skew of homogeneous degree $g$, i.e., $\epsilon \in \{+, -\}$.  
    
    In this context, we have the following definition. 
    
	\begin{definition}
		A \Gstar polynomial $f\in \mathcal{F}$ is a \Gstar identity of a $(G,*)$-algebra $A$ if 
		$\lambda(f)=0$ for any evaluation $\lambda$ of the variables of $f$ by elements of $A$ such that $\lambda(x_{i,g}^\varepsilon) \in A_g^\varepsilon$, for all $g\in G$. In this case, we write $f\equiv 0$ on $A$.
	\end{definition}

    We consider $\textnormal{Id}^{\sharp}(A)$  the set of all $\gi$-identities of $A$. Notice that $\textnormal{Id}^{\sharp}(A)$ is a $T_G^*$-ideal of $A$, i.e., an ideal invariant under all endomorphisms of $\mathcal{F}$ that preserve the grading and commute with the involution.
    
    For all $n\geq1$, we denote by $P_n^{\sharp}=\mbox{span}_F\{w_{\sigma(1)} \cdots w_{\sigma(n)}\mid \sigma\in S_n,\; w_i\in \{x_{i,g}^+,x_{i,g}^-\},\; g\in G \}$ the space of multilinear $\gi$-polynomials of degree $n$. In characteristic zero, it is well known that $\IdGstar{A}$ is generated by its multilinear $\gi$-identities. Thus, it makes sense to consider the quotient space $$P_n^\sharp(A)=\frac{P_n^\sharp}{P_n^\sharp \cap \textnormal{Id}^\sharp(A)},  \quad n\geq 1,$$ and define $c_n^{\sharp}(A) := \dim_F P_n^\sharp(A)$ as
   the $n$th $(G,*)$-codimension of $A$. 
   
   In the case where $A$ is a finite-dimensional $\gi$-algebra satisfying an ordinary polynomial identity, Oliveira, dos Santos and Vieira \cite{Lorena} showed that the sequence $c_n^\sharp(A),\,  n\geq 1$, either grows exponentially or is polynomially bounded. In the latter case, we say that $A$ has polynomial growth (of the sequence of $\gi$-codimensions).
    
   Recall that two \Gstar algebras $A$ and $B$ are  $T_G^*$-equivalent if $\textnormal{Id}^{\sharp} (A) = \textnormal{Id}^{\sharp}(B)$ and we write $A\sim_{T_G^*} B$. Moreover, we define $\mathcal{V}:=\textnormal{var}^{\sharp}(A)$ the $(G,*)$-variety generated by $A$, i.e., the class of all \Gstar algebras $B$ such that $\IdGstar{A}\subseteq \IdGstar{B}$.

    An important structural result about $\gi$-algebras with polynomial growth is given below.
    
    \begin{theorem}\label{estrutura} \cite{Maralice}
        Let $A$ be a finite-dimensional $\gi$-algebra over a field $F$ of characteristic zero. Then $c_n^\sharp (A)$, $n\geq 1$, is polynomially bounded if and only if $A \sim_{T_G^*} B$, where $B = B_1 \oplus \cdots \oplus B_m$, where $B_1,\ldots , B_m$ are finite-dimensional $\gi$-algebras over $F$ such that $\dim_F
B_i/J(B_i) \leq 1$ and $J(B_i)$ denotes the Jacobson radical of $B_i$, for all $i = 1, \ldots , m$.
    \end{theorem}

   In this paper, we are interested in unitary $\gi$-algebras. In this setting, proper $\gi$-polynomials play an important role. Recall that the commutator of length $n$ is defined inductively by $[x_{1}, \ldots , x_{n}] = [[x_{1}, \ldots , x_{n-1}], x_{n}],$ where $[x_{1},x_{2}] = x_{1}x_{2} - x_{2}x_{1}$. 
   A polynomial $f\in\mathcal{F}$ is called a proper $\gi$-polynomial if it is a linear combination of polynomials of type
\[
x_{i_1,g_2}^+\cdots x_{i_r,g_2}^+\cdots x^+_{j_1,g_k}\cdots x^+_{j_s,g_k}
x^-_{l_1,1}\cdots x^-_{l_t,1}\cdots x^-_{m_1,g_k}\cdots x^-_{m_u,g_k}
w_1\cdots w_v,
\]
where $w_1,\dots,w_v$ are long commutators in the variables of $X^+\cup X^-$. It is important to note that a proper $\gi$-polynomial consists of monomials in which symmetric variables of homogeneous degree $1$ appear only inside commutators.

One important aspect of the study of unitary $(G,*)$-algebras is that $\IdGstar{A}$ is generated by multilinear proper $\gi$-polynomials (see, for instance, \cite{Willer}).
We denote by $\Gamma^{\sharp}_n$ the space of proper multilinear $\gi$-polynomials of degree $n$ and define $\Gamma^{\sharp}_0=\spam_F\{1\}$.
The sequence of proper $\gi$-codimensions of $A$ is defined as
\[
\gamma^{\sharp}_n(A)=
\dim_F\frac{\Gamma^{\sharp}_n}{\Gamma^{\sharp}_n\cap\IdGstar{A}},\qquad n\geq 1.
\]

By \cite{Maralice}, the relation between the $\gi$-codimensions and the proper
$\gi$-codimensions of $A$ is given by
\begin{equation}\label{codimension}
c^{\sharp}_n(A)=\sum_{i=0}^n \binom{n}{i}\,\gamma^{\sharp}_i(A),\qquad n\ge 1.
\end{equation}

As we are concerned with unitary algebras with polynomial growth, the preceding result ensures the existence of an integer $t$ such that $\gamma_m^G(A)=0$ for every $m>t$. Therefore, in this case, we have 
$$c^{\sharp}_n(A)=\sum_{i=0}^t \binom{n}{i}\,\gamma^{\sharp}_i(A)= q n^t + q_1 n^{t-1} + \cdots = qn^{t}+\mathcal{O}(n^{t-1}) $$ which is a polynomial with rational coefficients. Moreover, the leading coefficient $q$ satisfies the following inequalities:
$$\frac{1}{t!} \le q \le \sum_{i=0}^{t}2^{t-i}|G|^{t-i}\frac{(-1)^{i}}{i!}.$$

Denote by $\Gamma_{({n_1}, \ldots , {n_{2k}})}$ the vector space of multilinear proper $(G,*)$-polynomials where the first $n_1$ variables are symmetric of homogeneous degree $g_1=1$, the next $n_2$ variables are skew of homogeneous degree $g_1=1$ and so on until the next $n_{2k-1}$ variables are symmetric of homogeneous degree $g_k$ and the last $n_{2k}$ variables are skew of homogeneous degree $g_k$. 

 In order to simplify the notation, we denote $({n_1}, \ldots , {n_{2k}})=({n_1}_{g_1^{+}},{n_{2}}_{g_1^-} ,\ldots , {n_{2k-1}}_{g_k^+}, {n_{{2k}}}_{{g_k}^-})$ where the zero values will be omitted. For instance, if $n=n_1+n_2+n_3+n_4+n_5+n_6$ where $n_1=n_5=1$, $n_2=n_3=n_6=0$ and $n_4=2$, then the $6$-tuple $(1,0,0,2,1,0)$ will be denoted by $(1_{1^+},2_{g_2^-}, 1_{g_3^+}).$

Note that the space $\Gamma_n^{\sharp}$ can be decomposed in the following way.

\begin{align}\label{Gamma_n}
    \Gamma_n^{\sharp}\cong \bigoplus_{n=n_1+\cdots+n_{2k}} \displaystyle \binom{n}{n_1, \ldots, n_{2k}}\Gamma_{({n_1}, \ldots , {n_{2k}})}
\end{align} where $\binom{n}{n_1,\ldots,n_{2k}}=\frac{n!}{n_1!\cdots n_{2k}!}$ denotes the multinomial coefficient.

Now, we consider a decomposition $n=n_1+\cdots+n_{2k}$ and define $\gamma_{n_1,\cdots,n_{2k}}(A) = \dim_F \Gamma_{n_1,\ldots,n_{2k}} (A)$, the proper $(n_1,\ldots,n_{2k})$-codimension, where
$$\Gamma_{n_1,\ldots,n_{2k}} (A)=\frac{\Gamma_{n_1,\ldots,n_{2k}}}{\Gamma_{n_1,\ldots,n_{2k}} \cap \Idd(A)}.$$
Consequently, we obtain

\begin{equation}\label{proper_codimension}
    \gamma_n^{\sharp}(A)=\displaystyle \sum _{n=n_1+\cdots+n_{2k}} \binom{n}{n_1,\ldots,n_{2k}}\gamma_{n_1,\ldots,n_{2k}}(A).
\end{equation}

It is well known that there exists a natural left action of $S_{n_1, \ldots,n_{2k}} := S_{n_1} \times \cdots \times S_{n_{2k}}$ on $P_{n_1, \ldots,n_{2k}}$ , where $S_{2i-1}$ acts by permuting the symmetric variables of homogeneous degree $g_{i}$ and $S_{n_{2i}}$ acts by permuting the skew variables of homogeneous degree $g_{i}$, for $1 \leq i \leq k$. It follows that $\Gamma_{n_1,\ldots,n_{2k}}$ is an $S_{n_1, \ldots,n_{2k}}$-submodule of $P_{n_1, \ldots,n_{2k}}$, and we consider $\chi(\Gamma_{n_1,\ldots,n_{2k}} )$ its character.
Since $\Gamma_{n_1,\ldots,n_{2k}} \cap \Idd(A)$ is invariant under this action, the space $\Gamma_{n_1,\ldots,n_{2k}} (A)$ inherits a structure of left $S_{n_1, \ldots,n_{2k}}$-module. We denote its character by $\pi_{n_1,\ldots,n_{2k}}(A)$, called the proper $(n_1,\ldots,n_{2k})$-cocharacter of $A$.
By complete reducibility, we can decompose $\pi_{n_1,\ldots,n_{2k}}(A)$ into irreducible characters as follows:

\begin{equation}\label{cocharacter}
    \pi_{n_1,\ldots,n_{2k}}(A)=\underset{(\lambda_1,\ldots, \lambda_{2k}) \vdash (n_1,\ldots,n_{2k})}{\sum} {{m}}_{\lambda_1,\ldots, \lambda_{2k}} \chi_{\lambda_1} \otimes \cdots \otimes \chi_{\lambda_{2k}}
\end{equation} where $(\lambda_1,\ldots, \lambda_{2k})$ is a multipartition of $ (n_1,\ldots,n_{2k})$ and $\chi_{\lambda_1} \otimes \cdots \otimes \chi_{\lambda_{2k}}$ is the irreducible $S_{n_1, \ldots,n_{2k}}$-character associated to such multipartition. The degree of the irreducible $S_{n_1, \ldots , n_{2k} }$ -character $\chi_{\lambda_1} \otimes \cdots \otimes \chi_{\lambda_{2k}}$ is given by $d_{\lambda_1} 
\cdots d_{\lambda_{2k}}$, where $d_{\lambda_i}$  is the degree of the irreducible $S_{n_i}$-character $\chi_{\lambda_i}$ given by the Hook Formula \cite[Theorem 3.10.2]{Sagan}. 

Now, we can express the $n$-th codimension of $A$ through the relations (\ref{codimension}), (\ref{proper_codimension}) and (\ref{cocharacter}) as follows 
\begin{equation}\label{codimension_cocharacters}
    c_n^\sharp(A)=\displaystyle\sum_{i=0}^{n}\,\,\,\, \displaystyle \sum _{i=n_1+\cdots+n_{2k}} \binom{n}{n_1,\ldots,n_{2k}}\pi_{n_1,\ldots,n_{2k}}(A)(1).
\end{equation}

 In order to determine the decomposition of the proper $(n_1, \ldots, n_{2k})$-cocharacters of $A$ we will make use of the representation theory of $GL_m$-modules, which provides an effective method to compute such multiplicities. In summary, it is well known that the corresponding irreducible $GL_m^{2k}$-modules are generated by the so called proper highest weight vectors (proper h.w.v.'s) $f_{\lambda_1,\ldots,\lambda_{2k}}$ associated to a multipartition $(\lambda_1, \ldots,\lambda_{2k}) \vdash (n_1,\ldots,n_{2k})$. We assume that the reader is familiar with this subject and we suggest \cite[Section 12.4]{Drensky} for more details.

For an explicit construction of such proper h.w.v.’s, the authors recommend \cite[Section 3]{Malu}, which treats the case of $(\mathbb{Z}_2,*)$-algebras. In the case of a $\gi$-algebra, this construction follows this theory step by step.  We present the main result on the computation of these multiplicities below. 

\begin{proposition}\label{prop:nonzero_multiplicities}
    Let $A$ be a unitary $\gi$-algebra with proper $(n_1,\ldots,n_{2k})$-cocharacter as in (\ref{cocharacter}). Then $m_{\lambda_1,\ldots, \lambda_{2k}}\neq 0$ if and only if there exists a proper h.w.v. $f_{\lambda_1,\ldots,\lambda_{2k}}$ associated to $(\lambda_1, \ldots,\lambda_{2k}) \vdash (n_1,\ldots,n_{2k})$ such that $f_{\lambda_1,\ldots,\lambda_{2k}} \notin \Idd(A)$. Moreover, $m_{\lambda_1,\ldots,\lambda_{2k}}$ is equal to the maximal number of proper h.w.v.'s associated to $(\lambda_1, \ldots,\lambda_{2k}) \vdash (n_1,\ldots,n_{2k})$ which are linearly independent modulo $\Idd(A)$.
\end{proposition} \black 

\begin{remark}\label{remark}
     Let $A$ be a unitary $(G,*)$-algebra with proper $(n_1,\ldots,n_{2k})$-cocharacter as in (\ref{cocharacter}) and let $B$ be a unitary $\gi$-algebra such that $B \in \varGstar(A)$ with $(n_1,\ldots,n_{2k})$-cocharacter given by
\begin{equation*}
		\pi_{n_1,\ldots,n_{2k}}(B)=\underset{(\lambda_1,\ldots, \lambda_{2k}) \vdash (n_1,\ldots,n_{2k})}{\sum} {\tilde{m}}_{\lambda_1,\ldots, \lambda_{2k}} \chi_{\lambda_1} \otimes \cdots \otimes \chi_{\lambda_{2k}}.
	\end{equation*}
    Since $\Gamma_{n_1,\ldots,n_{2k}} (B)$ can be embedded into $\Gamma_{n_1,\ldots,n_{2k}} (A)$ for all $n = n_1 + \cdots + n_{2k}$, we have $\tilde{m}_{\lambda_1,\ldots, \lambda_{2k}} \leq {m}_{\lambda_1,\ldots, \lambda_{2k}} $ for every multipartition $(\lambda_1,\ldots, \lambda_{2k}) \vdash (n_1,\ldots ,n_{2k})$.
\end{remark}
 
As the primary goal of this work is to investigate unitary varieties with quadratic growth, we conclude this section by explicitly presenting the decomposition of $\Gamma_n^\sharp$ into homogeneous subspaces for $n\in\{1,2\}$, as described in (\ref{Gamma_n}). 
\begin{equation}\label{decomposition_Gamma}
     \Gamma_1^{\sharp} \cong    \displaystyle \bigoplus_{s^\delta\in \mathcal{I}_1} \Gamma_{(1_{s^\delta})}, \quad \Gamma_2^{\sharp}  \cong 
\Gamma_{(2_{1^+})} \bigoplus_{s^\delta\in \mathcal{I}_1} (
\Gamma_{(2_{s^\delta})} \oplus
\Gamma_{(1_{1^+},1_{s^\delta})})\bigoplus_{\substack{ s^\delta \in \mathcal{I}_{1},t^\epsilon \in \mathcal{I}_2 \\ t^\epsilon\neq  s^\delta }}2
\Gamma_{(1_{s^\delta},1_{t^\epsilon}),}
\end{equation} where  $\mathcal{I}_1=\{1^-,g^+,g^- \ | \ g\in G^\times\}$ and $  \mathcal{I}_2=\mathcal{I}_1-\{1^-\}$.   Moreover, we exhibit a table with the information of each proper cocharacter associated with the previous decomposition. 

In the following, we recall that $y\circ z$ denotes the Jordan product $yz+zy$.

 \begin{longtable}{lllc}
\endfirsthead
\endhead
\endfoot
\endlastfoot
\toprule

\textbf{$\Gamma_{(n_1, \ldots, n_{2k})}$} 
& \textbf{proper $(n_1, \ldots, n_{2k})$-cocharacters} 
& \textbf{proper h.w.v's }
& \textbf{multiplicity} \\ 
\midrule

$\Gamma_{(1_{s^-})}$
& $\chi_{((1)_{s^-})}$
& $x^-_s$
& $1$ \\[2mm]

$\Gamma_{(1_{t^+})}$
& $\chi_{((1)_{t^+})}$
& $x^+_t$
& $1$ \\[2mm]

$\Gamma_{(2_{1^+})}$
& $\chi_{((1,1)_{1^+})}$
& $[x^+_{1,1},x^+_{2,1}]$
& $1$ \\[2mm]

$\Gamma_{(2_{t^+})}$
& $\chi_{((2)_{t^+})}+ \chi_{((1,1)_{t^+})}$
& $x^+_{1,t} \circ x^+_{2,t},\; [x^+_{1,t},x^+_{2,t}]$
& $1,\;1$ \\[2mm]

$\Gamma_{(2_{s^-})}$
& $\chi_{((2)_{s^-})}+ \chi_{((1,1)_{s^-})}$
& $x^-_{1,s} \circ x^-_{2,s},\; [x^-_{1,s},x^-_{2,s}]$
& $1,\;1$ \\[2mm]

$\Gamma_{(1_{1^+},1_{t^\epsilon})}$
& $\chi_{((1)_{1^+})} \otimes \chi_{((1)_{t^\epsilon})}$
& $[x^+_{1,1},x^\epsilon_{2,t}]$
& $1$ \\[2mm]

$\Gamma_{(1_{s^\delta},1_{t^\epsilon})}$
& $2\chi_{((1)_{s^\delta})}\otimes \chi_{((1)_{t^\epsilon})}$
& $[x^\delta_{1,s},x^\epsilon_{2,t}],\; x^\delta_{1,s}x^\epsilon_{2,t}$
& $2$ \\ [2mm]

\toprule

\caption{Proper $(n_1,\ldots,n_{2k})$-cocharacters of $\Gamma_{(n_1,\ldots, n_{2k})}$}
\label{tabela_mult}

\end{longtable}

\section{Minimal unitary $\gi$-varieties with quadratic growth}

Recall that a $(G,*)$-algebra $A$ generates a minimal variety of polynomial growth $n^t$ if $c_n^\sharp(A)\approx qn^t$, and for any proper subvariety $\mathcal{U}\subsetneq \textnormal{var}^\sharp(A)$ we have $c_n^\sharp(\mathcal{U})\approx q'n^p$ for some $p<t$. In \cite{Wesley3}, minimal $\gi$-varieties with quadratic codimension growth were classified. The algebras generating these varieties play an important role in the classification of unitary $(G,*)$-varieties of quadratic growth, as we will see at the end of this paper.

In this section, we present the $T_G^*$-ideals, the $\gi$-codimensions and the proper $(n_1,\ldots,n_{2k})$-cocharacters of unitary algebras generating minimal varieties.

 For $m \geq 2$ denote the $m\times m$ identity matrix by $I_{m}$ and let $E_1 = \sum\limits_{i = 1}^{m-1} e_{i,i+1} \in UT_{m}$, where $UT_m$ denotes the algebra of upper triangular matrices of size $m$. Define the commutative subalgebra of $UT_m$ given by
	$$C_{m}= \{\alpha I_m + \underset{1\leq i<m}{\sum} \alpha_i E_1^i\mid \alpha, \alpha_i \in F\}.$$

 For this algebra, we fix $g \in G$ and consider the only $G$-grading such that
    $$I_m \in (C_m)_1, \quad E_1^i \in (C_m)_ {g^i},  \quad \textnormal{for} \ i=1,\ldots,m-1$$
 and the nontrivial involution given by 
 \begin{equation}
     (\alpha I_{m} + \underset{1\leq i<m}{\sum} \alpha_i E_1^i)^*= \alpha I_{m} + \underset{1\leq i<m}{\sum} (-1)^i\alpha_i E_1^i.
 \end{equation}

    Since $C_m$ is a commutative algebra, we denote by $C_{m}^{g}$ the algebra $C_m$ endowed with the previous $G$-grading and trivial involution. Moreover, we consider $C_{m,*}^{g}$ as the algebra $C_m$ with the previous grading and the involution above.

\begin{lemma}\label{T-ideal_C_3} \cite{Wesley2, RN,Nascimento}   Consider $g,h,s\in G^\times$ with $|g|>2$, $|h|=2$ and $t\in \{1,g,g^2\}$. Then,
		\begin{enumerate}

             \item[1)] $c_n^{\sharp}(B_1)=1+n$, $c_n^{\sharp}(B_2)=1+2n+\displaystyle\binom{n}{2}$ and $c_n^{\sharp}(B_3)=1+n+\displaystyle\binom{n}{2}$, for all $B_1\in \{C_2^s, C_{2,*}^s\}$, $B_2\in \{C_3^g, C_{3,*}^g\}$ and $B_3\in \{C_3^h, C_{3,*}^h, C_{3,*}^1\}$.

             \item[2)] The nonzero proper $(n_1, \ldots, n_{2k})$-cocharacters of the previous algebras are given by Table \ref{cocharacterofc3}.

  \begin{longtable}{ll}
\endfirsthead
\endhead
\endfoot
\endlastfoot
\toprule

$C_{2}^{s}$
& $\chi_{((1)_{s^+})}$
 \\[2mm]

$C_{2,*}^{s}$
& $\chi_{((1)_{s^-})}$
 \\[2mm]

$C_{2,*}^{1}$
& $\chi_{((1)_{1^-})}$
 \\[2mm]

$C_{3}^{g}$
& $\chi_{((1)_{g^+})},\quad \chi_{((1)_{{(g^2)}^+})}, \quad \chi_{((2)_{g^+})}$ \\[2mm]

$C_{3}^{h}$
& $\chi_{((1)_{h^-})} , \quad \chi_{((2)_{h^-})}$ \\[2mm]

$C_{3,*}^{1}$
& $\chi_{((1)_{1^-})},\quad \chi_{((2)_{1^-})}$ \\[2mm]

$C_{3,*}^{h}$
& $\chi_{((1)_{h^+})}, \quad \chi_{((2)_{h^+})}$ \\[2mm]

$C_{3,*}^{g}$
& $\chi_{((1)_{g^-})},\quad  \chi_{((1)_{{(g^2)}^-})},\quad  \chi_{((2)_{g^-})}$ \\ [2mm]

\toprule

\caption{Proper  cocharacters of $C_2$ and $C_3$.}
\label{cocharacterofc3}

\end{longtable}

		\end{enumerate}
	\end{lemma}
\begin{proof} 

The first item follows from the results in \cite{Wesley2, RN,Nascimento}.
In order to prove the item $2)$, we make use of Proposition \ref{prop:nonzero_multiplicities}. For instance, assume that $|g|>2$ and consider the $\gi$-algebra $C_{3,*}^g$. Since $C_{3,*}^g$ has quadratic codimension growth then $\Gamma^\sharp_ n \subset \Idd(C_{3,*}^g)$ for all $n\geq 3$. Furthermore, since $\supp^*(C_{3,*}^g)=\{1^+,g^-,{(g^2)}^-\}$, by Proposition \ref{prop:nonzero_multiplicities}, we have that $m_{((1)_{g^-})}$ and $m_{((1)_{{g^2}^-})}$ are nonzero. For $n=2$, we note that $x_{1,g}^- \circ x_{2,g}^-\notin  \textnormal{Id}^\sharp(C_{3,*}^g)$ and so, by Proposition \ref{prop:nonzero_multiplicities}, $m_{((2)_{g^-})}
\neq 0$. Finally, by (\ref{codimension_cocharacters}) we obtain
    $$c_n^\sharp(C_{3,*}^g)\geq 1+2n +\binom{n}{2}.$$
    Since, by \cite{Wesley2}, we have $c_n^\sharp(C_{3,*}^g)=1+2n +\binom{n}{2}$, then the only nonzero proper $(n_1, \ldots, n_{2k})$-cocharacters of 
    $C_{3,*}^g$ are $\chi_{((1)_{g^-})}$, $\chi_{((1)_{{g^2}^-})}$ and $\chi_{((2)_{g^-})}$.

The other items are proved similarly.
\end{proof}
  
  Throughout this section, we present several tables of this form, and in each case the proof follows a similar pattern: we exhibit the corresponding proper highest weight vector and verify that it is not an identity of the algebra.

For $m\geq 2$, consider $E = \sum\limits_{i = 2}^{m-1} e_{i,i+1} + e_{2m-i,2m-i+1} \in UT_{2m}$ and define the following subalgebras of $UT_{2m}$	$$N_m = \mbox{span}_F \{I_{2m}, E, \ldots , E^{m-2}; e_{12} - e_{2m-1,2m}, e_{13}, \ldots , e_{1m}, e_{m+1,2m}, e_{m+2,2m}, \ldots , e_{2m-2,2m}\},$$
		$$ U_m = \mbox{span}_F \{ I_{2m}, E, \ldots , E^{m-2}; e_{12} + e_{2m-1,2m}, e_{13}, \ldots , e_{1m}, e_{m+1,2m}, e_{m+2,2m}, \ldots , e_{2m-2,2m}\}.$$
    
   Let $g \in G$ and consider $N_{3,*}^ {g}$ and $ U_{3,*}^ {g}$ as the algebras $N_3$ and $U_3$, respectively, with elementary $G$-grading induced by $(1, g,g,1,1,g)$ and reflection involution.

From now on, the element $r\in G$ will be used exclusively to denote an element outside the support of the grading, that is, $r\notin \supp(A)$.

\begin{lemma} \label{T-ideal_N_e_U} \cite{Wesley2,Nascimento} For $g\in G^\times$ we have
    
		\begin{enumerate}    
            \item[1)] $\textnormal{Id}^{\sharp}(U_{3,*}^ {g})=\langle x_{1,1}^-, x_{1,g}x_{2,g}, [x_{1,1}^+,x_{2,1}^-], x_{1,r}^\epsilon \rangle_{T_G^*}$ and  $\textnormal{Id}^{\sharp}(N_{3,*}^ {g})=\langle x_{1,1}^-, x_{1,g}x_{2,g}, [x_{1,g}^+,{x_{1,1}^+}], x_{1,r}^\epsilon \rangle_{T_G^*}$. 
            
             \item[2)] $c_n^{\sharp}(U_{3,*}^1)=1+n+\displaystyle \binom{n}{2}$, $c_n^{\sharp}(N_{3,*}^{1})=1+n+\displaystyle2\binom{n}{2}$ and $c_n^{\sharp}(B)=1+ 2n+ 2\displaystyle\binom{n}{2}$, for all $B\in \{U_{3,*}^{ {g}},N_{3,*}^{ {g}} \}.$

             \item[3)] The nonzero proper $(n_1, \ldots, n_{2k})$-cocharacters of the previous algebras are given by Table \ref{table_N3_U3}.

             \begin{longtable}{ll}
\endfirsthead
\endhead
\endfoot
\endlastfoot

\toprule

$U_{3,*}^{1}$
& $\chi_{((1)_{1^-})},\quad \chi_{((1,1)_{1^+})}$ \\[2mm]

$U_{3,*}^{g}$
& $\chi_{((1)_{g^+})},\quad  \chi_{((1)_{g^-})},\quad \chi_{((1)_{1^+})}\otimes \chi_{((1)_{g^+})}$ \\[2mm]

$N_{3,*}^{1}$
& $\chi_{((1)_{1^-})},\quad \chi_{((1)_{1^+})}\otimes \chi_{((1)_{1^-})}$ \\[2mm]

$N_{3,*}^{g}$
& $\chi_{((1)_{1^-})},\quad \chi_{((1)_{g^+})},\quad \chi_{((1)_{1^+})}\otimes \chi_{((1)_{g^-})}$ \\ [2mm] \toprule

\caption{Proper cocharacters of $U_3$ and $N_3$.}
\label{table_N3_U3}

\end{longtable}
		\end{enumerate}
\end{lemma} \vspace{-0.6cm}

For $m\geq 2$, consider  the subalgebra $\mathcal{G}_m=\langle 1,e_1,\ldots, e_m \mid e_ie_j=-e_je_i  \rangle$ of the infinite-dimensional Grassmann algebra. For this algebra, we consider the involutions $\psi, \tau$ and $\gamma$ defined, respectively, by	\begin{align*}
		\psi( e_i)=e_i,\,\,\,\,\, \tau( e_i)=-e_i,\,\,\,\,\, \gamma( e_i)=(-1)^ie_i,\mbox{ for all }i=1,\ldots, m.
	\end{align*}

    For $g,h\in G$, we denote by $\mathcal{G}_{2,*}^{g,h}$ the algebra $\mathcal{G}_2$ with involution $*\in \{\psi, \tau, \gamma\}$ and the $G$-grading determined by $$ 1\in (\mathcal{G}_{2,*}^{g,h})_1,\, e_1\in (\mathcal{G}_{2,*}^{g,h})_g ,\,  e_2\in (\mathcal{G}_{2,*}^{g,h})_h\mbox{ and } \,e_1e_2\in (\mathcal{G}_{2,*}^{g,h})_{gh} .$$

\begin{lemma}\label{T-ideal_G_2tau} \cite{Wesley2,Nascimento} For $g,h, s,u \in G$, $|g|>2$, $|h|=2$, $s,u \neq 1$ and  $su\neq1$ we have
\begin{itemize}       
    \item[1)] $\textnormal{Id}^{\sharp}(\mathcal{G}_{2,\tau}^{g,g})=\langle x_{1,1}^-, x_{1,g}^+, x_{1,g^2}^+, x_{1,r}^\epsilon  \rangle_{T_G^*}$.  
    
    \item[2)]$c_n^{\sharp}(B_1)=1+2n+\displaystyle\binom{n}{2}$, $c_n^{\sharp}(B_2)=1+3n+2\displaystyle\binom{n}{2}$, $c_n^{\sharp}(\mathcal{G}_{2,\tau}^{1,s} )=1+2n+2\displaystyle\binom{n}{2}$, $c_n^{\sharp}(\mathcal{G}_{2,\tau}^{1,1})=1+n+\displaystyle\binom{n}{2}$, for all $B_1\in \{ \mathcal{G}_{2,\tau}^{h,h}, \mathcal{G}_{2,\tau}^{g,g}\}$ and $B_2\in \{ \mathcal{G}_{2,\tau}^{s,u},\mathcal{G}_{2,\tau}^{g,g^{-1}} \}$. 

    \item[3)] The nonzero proper $(n_1, \ldots, n_{2k})$-cocharacters of the previous algebras are given by  Table \ref{tableG_2tau}.

\begin{longtable}{ll}
\endfirsthead
\endhead
\endfoot
\endlastfoot

\toprule 
$\mathcal{G}_{2,\tau}^{1,1}$
& $\chi_{((1)_{1^-})},\quad \chi_{((1,1)_{1^-})}$  \\[2mm]

$\mathcal{G}_{2,\tau}^{1,s}$
& $\chi_{((1)_{1^-})},\quad  \chi_{((1)_{s^-})}, \quad \chi_{((1)_{1^-})} \otimes \chi_{((1)_{s^-})}$ \\[2mm]

$\mathcal{G}_{2,\tau}^{h,h}$
& $\chi_{((1)_{1^-})},\quad  \chi_{((1)_{h^-})},\quad \chi_{((1,1)_{h^-})}$ \\[2mm]

$\mathcal{G}_{2,\tau}^{g,g}$
& $\chi_{((1)_{g^-})},\quad \chi_{((1)_{(g^2)^-})},\quad  \chi_{((1,1)_{g^-})}$ \\[2mm]

$\mathcal{G}_{2,\tau}^{g,g^{-1}}$
& $\chi_{((1)_{1^-})},\quad  \chi_{((1)_{g^-})},\quad \chi_{((1)_{(g^{-1})^-})},\quad \chi_{((1)_{g^-})} \otimes \chi_{((1)_{(g^{-1})^-})}$ \\[2mm]

$\mathcal{G}_{2,\tau}^{s,u}$
& $\chi_{((1)_{s^-})},\quad \chi_{((1)_{u^-})},\quad  \chi_{((1)_{su^-})},\quad \chi_{((1)_{s^-})} \otimes \chi_{((1)_{u^-})}$ \\ [2mm] \toprule

\caption{Proper cocharacters of $\mathcal{G}_{2}$ with $\tau$ involution.}
\label{tableG_2tau}
\end{longtable}
    \end{itemize}
\end{lemma} \vspace{-0.6cm}

\begin{lemma}\label{T-ideal_G_2gamma} \cite{Wesley2,Nascimento} For $g,h, s,u \in G$, $|g|>2$, $|h|=2$, $s,u \neq 1$ and  $su\neq1$ we have
\begin{itemize}   

\item[1)] $\textnormal{Id}^{\sharp}(\mathcal{G}_{2,\psi}^{g,g})=\langle x_{1,1}^-, x_{1,g}^-, x_{1,g^2}^+, x_{1,r}^\epsilon  \rangle_{T_G^*}.$ 

    \item[2)] $c_n^{\sharp}(B_1)=1+2n+2\displaystyle\binom{n}{2}$,   $c_n^{\sharp}(B_2)=1+3n+2\displaystyle\binom{n}{2}$ and $c_n^\sharp(B_3)=1+2n+\displaystyle\binom{n}{2}$, for all $B_1\in \{\mathcal{G}_{2,\gamma}^{1,s}, \mathcal{G}_{2,\gamma}^{h,h},  \mathcal{G}_{2,\gamma}^{g,g^{-1}} \}$, $B_2\in \{\mathcal{G}_{2,\gamma}^{g,g}, \mathcal{G}_{2,\gamma}^{s,u},\mathcal{G}_{2,\psi}^{g,g^{-1}}, \mathcal{G}_{2,\psi}^{s,u} \}$ and $B_3\in \{ \mathcal{G}_{2,\psi}^{g,g},  \mathcal{G}_{2,\psi}^{h,h}\}$.

    \item[3)] The nonzero proper $(n_1, \ldots, n_{2k})$-cocharacters  are given by Table \ref{tableG_2gamma_psi}.

\begin{longtable}{ll}
\endfirsthead
\endhead
\endfoot
\endlastfoot

\toprule

$\mathcal{G}_{2,\gamma}^{1,s}$
& $\chi_{((1)_{1^-})},\quad  \chi_{((1)_{s^+})},\quad \chi_{((1)_{1^-})} \otimes \chi_{((1)_{s^+})}$ \\[2mm]

$\mathcal{G}_{2,\gamma}^{h,h}$
& $\chi_{((1)_{h^+})},\quad  \chi_{((1)_{h^-})},\quad \chi_{((1)_{h^+})} \otimes \chi_{((1)_{h^-})}$ \\[2mm]

$\mathcal{G}_{2,\gamma}^{g,g}$
& $\chi_{((1)_{g^+})},\quad  \chi_{((1)_{g^-})},\quad  \chi_{((1)_{(g^2)^+})},\quad \chi_{((1)_{g^+})} \otimes \chi_{((1)_{g^-})}$ \\[2mm]

$\mathcal{G}_{2,\gamma}^{g,g^{-1}}$
& $\chi_{((1)_{g^-})},\quad  \chi_{((1)_{(g^{-1})^+})},\quad \chi_{((1)_{g^-})} \otimes \chi_{((1)_{(g^{-1})^+})}$ \\[2mm]

$\mathcal{G}_{2,\gamma}^{s,u}$
& $\chi_{((1)_{s^-})},\quad  \chi_{((1)_{u^+})},\quad  \chi_{((1)_{su^+})},\quad \chi_{((1)_{s^-})} \otimes \chi_{((1)_{u^+})}$ \\[2mm]

 $\mathcal{G}_{2,\psi}^{h,h}$  
& $\chi_{((1)_{1^-})},\quad  \chi_{((1)_{h^+})},\quad \chi_{((1,1)_{h^+})}$ \\[2mm]

$\mathcal{G}_{2,\psi}^{g,g}$
& $\chi_{((1)_{g^+})},\quad \chi_{((1)_{(g^2)^-})},\quad \chi_{((1,1)_{g^+})}$ \\[2mm]

$\mathcal{G}_{2,\psi}^{g,g^{-1}}$
& $\chi_{((1)_{1^-})},\quad  \chi_{((1)_{g^+})},\quad  \chi_{((1)_{(g^{-1})^+})},\quad \chi_{((1)_{g^+})} \otimes \chi_{((1)_{(g^{-1})^+})}$ \\[2mm]

$\mathcal{G}_{2,\psi}^{s,u}$
& $\chi_{((1)_{u^+})},\quad  \chi_{((1)_{s^+})},\quad \chi_{((1)_{su^-})},\quad \chi_{((1)_{s^+})} \otimes \chi_{((1)_{u^+})}$ \\[2mm] \toprule

\caption{Proper cocharacters of $\mathcal{G}_{2}$ with $\gamma$ and $\psi$ involutions.}
\label{tableG_2gamma_psi}

\end{longtable}
    \end{itemize}
\end{lemma} \vspace{-0.6cm}

    Now we introduce the following commutative subalgebra of $UT_4:$
    $$W= F(e_{11}+\cdots + e_{44})+F(e_{12}+e_{34})+F(e_{13}+e_{24})+Fe_{14}$$
	   and we equip $W$ with the trivial involution $\nu_1$ and the involutions $\nu_2$ and $\nu_3$ defined below
	$$ \begin{pmatrix}
		a& b & c & d \\
		0&a & 0 & c\\
		0&0 & a & b \\
		0&0 & 0 & a 
	\end{pmatrix}^{\nu_2}= \begin{pmatrix}
		a& -b & -c & d \\
		0&a & 0 & -c\\
		0&0 & a & -b \\
		0&0 & 0 & a 
	\end{pmatrix} \mbox{ and } \begin{pmatrix}
		a& b & c & d \\
		0&a & 0 & c\\
		0&0 & a & b \\
		0&0 & 0 & a 
	\end{pmatrix}^{\nu_3}= \begin{pmatrix}
		a& -b & c & -d \\
		0&a & 0 & c\\
		0&0 & a & -b \\
		0&0 & 0 & a 
	\end{pmatrix}.$$

    For $g,h \in G$, we denote by $W_*^{g,h}$ as the algebra $W$ with involution $*\in \{\nu_1, \nu_2,\nu_3\}$ and the only $G$-grading determined by 
$$e_{11}+\cdots +e_{44}\in (W_*^{g,h})_1, \,e_{12}+e_{34}\in (W_*^{g,h})_g,\,e_{13}+e_{24}\in (W_*^{g,h})_h\mbox{ and } e_{14}\in (W_*^{g,h})_{gh}. $$  

 Regarding the previous algebra, we have the following lemmas, which can be consulted in \cite{Wesley2}.

    \begin{lemma} \cite{Wesley2} Consider $g,s,u \in G^\times$ such that $|g|>2$, $s\neq u$ and $su\neq 1$. Then,
    \begin{enumerate}
        \item[1)] $c_n^{\sharp}(B_1)=1+3n+
            {2}\displaystyle\binom{n}{2}$ and $c_n^{\sharp}(B_2)=1+2n+\displaystyle {2}\binom{n}{2}$, for all $B_1\in\{W_{\nu_1}^{s,u}, W_{\nu_2}^{s,u}, W_{\nu_2}^{1,s}, W_{\nu_3}^{s,u}, W_{\nu_3}^{h,h},$ $ W_{\nu_3}^{1,s}, W_{\nu_3}^{g,g^{-1}}, W_{\nu_3}^{g,g} \}$ and $B_2\in \{W_{\nu_1}^{g,g^{-1}}, W_{\nu_2}^{g,g^{-1}}\}$.

        \item[2)] The nonzero proper $(n_1, \ldots, n_{2k})$-cocharacters of the previous algebras are given by  Table \ref{table_W}.
\begin{longtable}{ll}
\endfirsthead
\endhead
\endfoot
\endlastfoot
\toprule

$W_{\nu_1}^{s,u}$
& $\chi_{((1)_{s^+})},\quad  \chi_{((1)_{u^+})},\quad  \chi_{((1)_{su^+})},\quad \chi_{((1)_{s^+})} \otimes \chi_{((1)_{u^+})}$ \\[2mm]

$W_{\nu_1}^{g,g^{-1}}$
& $\chi_{((1)_{g^+})},\quad  \chi_{((1)_{(g^{-1})^+})},\quad \chi_{((1)_{g^+})} \otimes \chi_{((1)_{(g^{-1})^+})}$ \\
[2mm]

$W_{\nu_2}^{s,u}$
& $\chi_{((1)_{s^-})},\quad  \chi_{((1)_{u^-})},\quad  \chi_{((1)_{su^+})},\quad \chi_{((1)_{s^-})} \otimes \chi_{((1)_{u^-})}$ \\[2mm]

$W_{\nu_2}^{g,g^{-1}}$
& $\chi_{((1)_{g^-})},\quad  \chi_{((1)_{(g^{-1})^-})},\quad \chi_{((1)_{g^-})} \otimes \chi_{((1)_{(g^{-1})^-})}$ \\[2mm]

$W_{\nu_2}^{1,s}$
& $\chi_{((1)_{1^-})},\quad  \chi_{((1)_{s^+})},\quad \chi_{((1)_{s^-})},\quad \chi_{((1)_{1^-})} \otimes \chi_{((1)_{s^-})}$ \\ [2mm]

$W_{\nu_3}^{s,u}$
& $\chi_{((1)_{s^-})},\quad  \chi_{((1)_{u^+})},\quad  \chi_{((1)_{su^-})},\quad \chi_{((1)_{s^-})} \otimes \chi_{((1)_{u^+})}$ \\[2mm]

$W_{\nu_3}^{g,g}$
& $\chi_{((1)_{g^+})},\quad  \chi_{((1)_{g^-})},\quad  \chi_{((1)_{(g^2)^-})},\quad \chi_{((1)_{g^+})} \otimes \chi_{((1)_{g^-})}$ \\[2mm]

$W_{\nu_3}^{g,g^{-1}}$
& $\chi_{((1)_{1^-})},\quad  \chi_{((1)_{g^-})},\quad  \chi_{((1)_{(g^{-1})^+})},\quad \chi_{((1)_{g^-})} \otimes \chi_{(((1)_{(g^{-1})^+})}$ \\[2mm]

$W_{\nu_3}^{1,s}$
& $\chi_{((1)_{1^-})},\quad  \chi_{((1)_{s^+})},\quad  \chi_{((1)_{s^-})},\quad \chi_{((1)_{1^-})} \otimes \chi_{((1)_{s^+})}$ \\[2mm]

$W_{\nu_3}^{h,h}$
& $\chi_{((1)_{1^-})},\quad  \chi_{((1)_{h^+})},\quad  \chi_{((1)_{h^-})},\quad \chi_{((1)_{h^+})} \otimes \chi_{((1)_{h^-})}$ \\ [2mm]

\toprule

\caption{Proper cocharacters of $W$.}
\label{table_W}
\end{longtable}
    \end{enumerate}
    \end{lemma}\vspace{-0.6cm}

Now, we consider some direct sums of the previous algebras and their respectively $T_G^*$-ideals. In the following, $k$ and $v$ will run through all elements of $\{s,u\}$ and $\{g,g^{-1}\}$, respectively. Moreover, we denote by $q_a$ any element of the set $\{1,a\}$, for some $a\in G.$

\begin{lemma}\label{Direct_sum} For $s,u \in G^\times$ with $s \neq u$ and $g \in G$ with $|g|>2$ we have
\begin{enumerate} 
    \item[1)] $\IdGstar{\mathcal{G}_{2,\psi}^{s,u}\oplus W_{\nu_1}^{s,u}}= \langle x_{1,q_k}^-, x_{1,k}^+x_{2,k}^+,x_{1,k}^+x_{2,su}^\epsilon,
        x_{1,su}^\delta x_{2,su}^\epsilon,[x_{1,1}^+,x_{2,su}^\epsilon], x_{1,r}^\epsilon \rangle_{T_G^*}$

    \item[2)] $\IdGstar{\mathcal{G}_{2,\tau}^{s,u}\oplus W_{\nu_2}^{s,u}}= \langle x_{1,1}^-, x_{1,k}^+, x_{1,k}^-x_{2,k}^-, x_{1,k}^-x_{2,su}^\epsilon, x_{1,su}^\delta x_{2,su}^\epsilon, [x_{1,1}^+,x_{2,su}^\epsilon], x_{1,r}^\epsilon  \rangle_{T_G^*}$

    \item[3)] $\IdGstar{\mathcal{G}_{2,\gamma}^{s,u}\oplus W_{\nu_3}^{s,u}}= \langle x_{1,q_u}^-, x_{1,s}^+, x_{1,k}^\epsilon x_{2,k}^\epsilon , x_{1,s}^-x_{2,su}^\epsilon,
        x_{1,u}^+x_{2,su}^\epsilon, x_{1,su}^\delta x_{2,su}^\epsilon, [x_{1,1}^+,x_{2,su}^\epsilon], x_{1,r}^\epsilon \rangle_{T_G^*}$

    \item[4)] $\IdGstar{\mathcal{G}_{2,\psi}^{g,g^{-1}}\oplus W_{\nu_1}^{g,g^{-1}}}= \langle x_{1,v}^-, x_{1,v}^+x_{2,v}^+, x_{1,1}^-x_{2,1}^-, x_{1,1}^-x_{2,v}^\epsilon, [x_{1,1}^+,x_{2,1}^\epsilon], x_{1,r}^\epsilon \rangle_{T_G^*}$

    \item[5)] $\IdGstar{\mathcal{G}_{2,\tau}^{g,g^{-1}}\oplus W_{\nu_2}^{g,g^{-1}}}= \langle x_{1,v}^+, x_{1,v}^-x_{2,v}^-, x_{1,1}^-x_{2,1}^-, x_{1,1}^-x_{2,v}^-, [x_{1,1}^+,x_{2,1}^\epsilon], x_{1,r}^\epsilon
    \rangle_{T_G^*}$

     \item[6)] $\IdGstar{\mathcal{G}_{2,\gamma}^{g,g^{-1}}\oplus W_{\nu_3}^{g,g^{-1}}}= \langle x_{1,g}^+, x_{1,g^{-1}}^-, x_{1,v}^\epsilon x_{2,v}^\epsilon, x_{1,1}^-x_{2,q_g}^-,
        x_{1,1}^-x_{2,g^{-1}}^+, [x_{1,1}^+,x_{2,1}^\epsilon], x_{1,r}^\epsilon\rangle_{T_G^*}$

        \item[7)] $c_n^\sharp(B_1)=1+3n+{4}\displaystyle\binom{n}{2}$ and $c_n^\sharp(B_2) =1+4n+{4}\displaystyle\binom{n}{2}$, for all $B_1\in \{\mathcal{G}_{2,\psi}^{g,{g^{-1}}}\oplus W_{\nu_1}^{g,g^{-1}}, 
\mathcal{G}_{2,\tau}^{g,g^{-1}}\oplus W_{\nu_2}^{g,g^{-1}},  \mathcal{G}_{2,\gamma}^{g,g^{-1}}\oplus W_{\nu_3}^{g,g^{-1}}\}$ and $B_2\in \{\mathcal{G}_{2,\psi}^{s,u}\oplus W_{\nu_1}^{s,u}, \mathcal{G}_{2,\tau}^{s,u}\oplus W_{\nu_2}^{s,u}, \mathcal{G}_{2,\gamma}^{s,u}\oplus W_{\nu_3}^{s,u} \}$.

        \item[8)] The nonzero proper $(n_1, \ldots, n_{2k})$-cocharacters of the previous algebras are given by Table \ref{table_sum1}.
       \begin{longtable}{ll}
\endfirsthead
\endhead
\endfoot
\endlastfoot
\toprule

$\mathcal{G}_{2,\psi}^{s,u} \oplus W_{\nu_1}^{s,u}$ &
$\chi_{((1)_{s^+})},\quad \chi_{((1)_{u^+})},\quad \chi_{((1)_{su^-})},\quad \chi_{((1)_{su^+})},\quad 2\,\chi_{((1)_{s^+})} \otimes \chi_{((1)_{u^+})}$ \\[2mm]

$\mathcal{G}_{2,\tau}^{s,u} \oplus W_{\nu_2}^{s,u}$ &
$\chi_{((1)_{s^-})},\quad \chi_{((1)_{u^-})},\quad \chi_{((1)_{su^+})},\quad \chi_{((1)_{su^-})},\quad 2\,\chi_{((1)_{s^-})} \otimes \chi_{((1)_{u^-})}$ \\[2mm]

$\mathcal{G}_{2,\gamma}^{s,u} \oplus W_{\nu_3}^{s,u}$ &
$\chi_{((1)_{s^-})},\quad \chi_{((1)_{u^+})},\quad \chi_{((1)_{su^-})},\quad \chi_{((1)_{su^+})},\quad 2\,\chi_{((1)_{s^-})} \otimes \chi_{((1)_{u^+})}$ \\[2mm]

$\mathcal{G}_{2,\psi}^{g,g^{-1}} \oplus W_{\nu_1}^{g,g^{-1}}$ &
$\chi_{((1)_{1^-})},\quad \chi_{((1)_{g^+})},\quad \chi_{((1)_{(g^{-1})^+})},\quad 2\,\chi_{((1)_{g^+})} \otimes \chi_{((1)_{(g^{-1})^+})}$ \\[2mm]

$\mathcal{G}_{2,\tau}^{g,g^{-1}} \oplus W_{\nu_2}^{g,g^{-1}}$ &
$\chi_{((1)_{1^-})},\quad \chi_{((1)_{g^-})},\quad \chi_{((1)_{(g^{-1})^-})},\quad 2\,\chi_{((1)_{g^-})} \otimes \chi_{((1)_{(g^{-1})^-})}$ \\[2mm]

$\mathcal{G}_{2,\gamma}^{g,g^{-1}} \oplus W_{\nu_3}^{g,g^{-1}}$ &
$\chi_{((1)_{1^-})},\quad \chi_{((1)_{g^-})},\quad \chi_{((1)_{(g^{-1})^+})},\quad 2\,\chi_{((1)_{g^-})} \otimes \chi_{((1)_{(g^{-1})^+})}$ \\[2mm]
\toprule
\caption{Proper cocharacters of some sums of $(G,*)$-algebras.}
\label{table_sum1} 
\end{longtable}
\end{enumerate}
\end{lemma} \vspace{-0.9cm}
\begin{proof} 
We consider only item 1, since the other cases follow similarly. Let $I$ be the $\gi$-ideal given by 
$$I=\langle x_{1,q_k}^-, x_{1,k}^+x_{2,k}^+,x_{1,k}^+x_{2,su}^\epsilon,
        x_{1,su}^\delta x_{2,su}^\epsilon,[x_{1,1}^+,x_{2,su}^\epsilon], x_{1,r}^\epsilon \rangle_{T_G^*}.$$ 
It is clear that $I\subseteq \IdGstar{\mathcal{G}_{2,\psi}^{s,u}\oplus W_{\nu_1}^{s,u}}$.
For the opposite inclusion, let $f \in \IdGstar{\mathcal{G}_{2,\psi}^{s,u}\oplus W_{\nu_1}^{s,u}}$ be a multilinear and multihomogeneous proper identity. Notice that, for $n\geq 3$, $\Gamma_n^\sharp \subseteq I$ and since  $\mathcal{G}_{2,\psi}^{s,u}\oplus W_{\nu_1}^{s,u}$ has quadratic growth, we have
$$ \Gamma_{n}^\sharp= \Gamma_n^{\sharp} \cap \IdGstar{\mathcal{G}_{2,\psi}^{s,u}\oplus W_{\nu_1}^{s,u}}\subseteq I,\quad \text{for all} \quad n \geq 3. $$

Moreover, it is straightforward to verify that
$$ \Gamma_1^{\sharp}\cap \IdGstar{\mathcal{G}_{2,\psi}^{s,u}\oplus W_{\nu_1}^{s,u}}=\text{span}_F\{x_{1,1}^-, x_{1,s}^-,x_{1,u}^-, x_{1,r}^\epsilon\} \subseteq I. $$
Therefore, it remains to consider the case where $f$ has degree 2. In this case, after reducing $f$ modulo $I$, we may assume that $f=\alpha x_{1,s}^+x_{2,u}^++\beta [x_{1,s}^+,x_{2,u}^+]$. 
Consider the evaluation $x_{1,s}^+ \mapsto (e_1,0)$ and $x_{2,u}^+ \mapsto (e_2,0)$, we obtain $\alpha +2\beta=0$.
Now, taking the evaluation $x_{1,s}^+ \mapsto (0,e_{12}+e_{34})$ and $x_{2,u}^+ \mapsto (0,e_{13}+e_{24})$, we obtain $\alpha +\beta=0$.
Therefore, we must have $\alpha=\beta=0$ and then $f\in I$. This proves that $I=\textnormal{Id}^\sharp(\mathcal{G}_{2,\psi}^{s,u}\oplus W_{\nu_1}^{s,u})$.

In order to prove item 8) for $\mathcal{G}_{2,\psi}^{s,u}\oplus W_{\nu_1}^{s,u}$, we consider again the Proposition \ref{prop:nonzero_multiplicities}. Notice that the proper h.w.v.'s $x_{1,s}^+, x_{1,u}^+, x_{1,su}^+, x_{1,su}^-$ and $[x_{1,s}^+,x_{2,u}^+]$ are not identities of $\mathcal{G}_{2,\psi}^{s,u}\oplus W_{\nu_1}^{s,u}$. Therefore, the nonzero multiplicities satisfy
$$m_\lambda = 1 \quad \textnormal{and} \quad  1 \leq m_{((1)_{s^+} ,(1)_{u^+})} \leq 2, \quad \textnormal{for all} ~ \lambda \in \{((1)_{s^+}) , ((1)_{u^+}), ((1)_{su^+}), ((1)_{su^-})\}.$$

Moreover, according to the above evaluation, there is no nonzero linear combination $\alpha x_{1,s}^-x_{2,u}^++\beta [x_{1,s}^-,x_{2,u}^+]$ that results in an element of $\textnormal{Id}^\sharp(\mathcal{G}_{2,\psi}^{s,u}\oplus W_{\nu_1}^{s,u})$. Therefore, by Proposition \ref{prop:nonzero_multiplicities} we have $m_{((1)_{s^+} ,(1)_{u^+})} = 2$. 

Finally, item 7) follows by equation (\ref{codimension_cocharacters}).
\end{proof}

Analogously to the previous theorem, the next result is established by the same method. Also, we recall that $q_a$ denotes any element of the set 
$\{1,a\}$, for some $a\in G.$

\begin{lemma}\label{Direct_sum} For $s,h \in G^\times$ and $g \in G$ with $|g|>2$, $|h|=2$ we have
\begin{enumerate}

    \item[1)] $\IdGstar{\mathcal{G}_{2,\gamma}^{g,g}\oplus W_{\nu_3}^{g,g}}= \langle x_{1,1}^-, x_{1,g}^\epsilon x_{2,g}^\epsilon ,x_{1,g}^\sigma x_{2,g^2}^\epsilon, x_{1,g^2}^\epsilon x_{2,g^2}^\epsilon,[x_{1,1}^+x_{2,q_{g^2}}^\epsilon], x_{1,r}^\epsilon  \rangle_{T_G^*}$

    \item[3)] $\IdGstar{\mathcal{G}_{2,\gamma}^{h,h}\oplus W_{\nu_3}^{h,h}}= \langle x_{1,1}^-x_{2,q_h}^-,
    x_{1,1}^-x_{2,h}^+,
    x_{1,h}^\epsilon x_{2,h}^\epsilon, [x_{1,1}^+,x_{2,q_h}^\epsilon],  x_{1,r}^\epsilon \rangle_{T_G^*}$

        \item[4)] $\IdGstar{\mathcal{G}_{2,\gamma}^{1,s}\oplus W_{\nu_3}^{1,s}}= \langle x_{1,1}^-x_{2,q_{s}}^-,x_{1,1}^-x_{2,s}^+, 
        x_{1,s}^+x_{2,s}^\epsilon, x_{1,s}^-x_{2,s}^-, [x_{1,1}^+,x_{2,q_s}^\epsilon],  x_{1,r}^\epsilon \rangle_{T_G^*}$

        \item[5)] $\IdGstar{\mathcal{G}_{2,\tau}^{1,g}\oplus W_{\nu_2}^{1,g}}= \langle x_{1,1}^-x_{2,1}^-,x_{1,1}^-x_{1,g}^+, 
        x_{1,g}^\sigma x_{2,g}^\epsilon,[x_{1,1}^+x_{2,q_{g}}^\epsilon], x_{1,r}^\epsilon \rangle_{T_G^*}$.

        \item[6)] $c_n^\sharp(\mathcal{G}_{2,\gamma}^{g,g}\oplus W_{\nu_3}^{g,g})=1+4n+{4}\displaystyle\binom{n}{2}$ and for all $B\in \{\mathcal{G}_{2,\gamma}^{h,h}\oplus W_{\nu_3}^{h,h}, 
\mathcal{G}_{2,\gamma}^{1,s}\oplus W_{\nu_3}^{1,s}, 
\mathcal{G}_{2,\tau}^{1,s}\oplus W_{\nu_2}^{1,s}\}$, we have $c_n^\sharp(B) =1+3n+{4}\displaystyle\binom{n}{2}$.

        \item[7)] The nonzero proper $(n_1, \ldots, n_{2k})$-cocharacters of the previous algebras are given by Table \ref{table_sum2}.
       \begin{longtable}{llc}
\endfirsthead
\endhead
\endfoot
\endlastfoot
\toprule

$\mathcal{G}_{2,\gamma}^{g,g} \oplus W_{\nu_3}^{g,g}$ &
$\chi_{((1)_{g^+})},\quad \chi_{((1)_{g^-})},\quad \chi_{((1)_{(g^2)^-})},\quad \chi_{((1)_{(g^2)^+})},\quad 2\,\chi_{((1)_{g^+})} \otimes \chi_{((1)_{g^-})}$ \\[2mm]

$\mathcal{G}_{2,\gamma}^{h,h} \oplus W_{\nu_3}^{h,h}$ &
$\chi_{((1)_{1^-})},\quad \chi_{((1)_{h^+})},\quad \chi_{((1)_{h^-})},\quad 2\,\chi_{((1)_{h^+})} \otimes \chi_{((1)_{h^-})}$ \\[2mm]

$\mathcal{G}_{2,\gamma}^{1,s} \oplus W_{\nu_3}^{1,s}$ &
$\chi_{((1)_{1^-})},\quad \chi_{((1)_{s^+})},\quad \chi_{((1)_{s^-})},\quad 2\,\chi_{((1)_{1^-})} \otimes \chi_{((1)_{s^+})}$ \\[2mm]

$\mathcal{G}_{2,\tau}^{1,s} \oplus W_{\nu_2}^{1,s}$ &
$\chi_{((1)_{1^-})},\quad \chi_{((1)_{s^+})},\quad \chi_{((1)_{s^-})},\quad 2\,\chi_{((1)_{1^-})} \otimes \chi_{((1)_{s^-})}$ \\ [2mm]

\toprule

\caption{Proper cocharacters of some sums of $(G,*)$-algebras.}
\label{table_sum2}
\end{longtable}

\end{enumerate}
\end{lemma}

\section{Characterizing varieties with nonzero multiplicities}

The aim of this section is to classify the unitary $(G,*)$-varieties with quadratic codimension growth. To this end, we establish a precise correspondence between the nonzero multiplicities appearing in the decomposition of the proper \((n_1, \ldots, n_{2k})\)-cocharacters and the minimal varieties contained in them.

Motivated by Theorem~\ref{estrutura}, from now on we assume that \( A \) is a unitary $\gi$-algebra with quadratic codimension growth and of type \( F + J(A) \). In this case, we have \( \Gamma_n ^
\sharp \subseteq \IdGstar{A} \), for all \( n \ge 3 \). Hence, our analysis is restricted to the subspaces \( \Gamma_1^\sharp \) and \( \Gamma_2^\sharp \), which are completely determined by their multihomogeneous components.

For $n=n_1+\cdots +n_{2k}$, with $n\in\{1,2\}$, we consider the decomposition of $\pi_{n_1,\ldots,n_{2k}}(A)$ as follows

$$ \pi_{n_1,\ldots,n_{2k}}(A)=\underset{(\lambda_1,\ldots, \lambda_{2k}) \vdash (n_1,\ldots,n_{2k})}{\sum} {{m}}_{\lambda_1,\ldots, \lambda_{2k}} \chi_{\lambda_1} \otimes \cdots \otimes \chi_{\lambda_{2k}}.$$

 It is well known that $\Gamma_{n_1,\ldots,n_{2k}}(A)$ can be seen as a direct summand of $\Gamma_{n_1,\ldots,n_{2k}},$ as an \(S_{n_1,\ldots,n_{2k}}\)-module. Therefore, the multiplicities appearing in the 
\((n_1,\ldots,n_{2k})\)-cocharacter of \(\Gamma_{n_1,\ldots,n_{2k}}(A)\) 
do not exceed the corresponding multiplicities in the 
\((n_1,\ldots,n_{2k})\)-cocharacter of \(\Gamma_{n_1,\ldots,n_{2k}}\). Hence, by Table, \ref{tabela_mult} we have
\begin{equation} 
\label{multiplicities}
    0\leq m_{\lambda }\leq 1 \quad \mbox{ and }\quad\quad 0\leq m_{((1)_{s^\delta},(1)_{t^\epsilon})}\leq 2,
\end{equation} for all $\lambda \in \mathcal{S}=\{((1)_{u^{\epsilon_{1}}}),((2)_{u^{\epsilon_1}}),(1,1)_{u^{\epsilon_2}},((1)_{1^+},(1)_{u^{\epsilon_1}})\mid u\in G, \epsilon_i\in \{+,-\},  u^{\epsilon_1}
\neq 1^+\}$ and $s^\delta\neq t^\epsilon \neq 1^+$. Moreover, since $A$ has quadratic growth, we must have $m_{\lambda}\neq 0$, for some  multipartition $\lambda$ of $2$.

In the following lemmas, we analyze all admissible values of the multiplicities introduced above and provide a characterization of the corresponding varieties in terms of these values.

\begin{lemma}\label{lineares1} For $g,h\in G$ with $h \neq 1$, we have
    \begin{enumerate}
        \item $m_{((1)_{h^+})}\neq 0$ if and only if $C_{2}^h \in \varGstar(A)$.
        \item $m_{((1)_{g^-})}\neq 0$ if and only if $C_{2,*}^g \in \varGstar(A)$.
    \end{enumerate}
\end{lemma}
\begin{proof}

    Assume that $m_{((1)_{s^\epsilon})} \neq 0$ for some $s^\epsilon \in \{h^+, g^-\}$  with $g,h\in G$ and $h \neq 1$. By Proposition \ref{prop:nonzero_multiplicities} and Table \ref{tabela_mult}, there exists a nonzero element $a \in J(A)_s^\epsilon$. Let $R$ be the $\gi$-subalgebra of $A$ generated by $1_F$ and $a$, and let $I$ be the $T_G^*$-ideal of $A$ generated by $a^2$. The map $\varphi: R/I\rightarrow C_2$ given by $$\overline{1_F}\mapsto e_{11}+e_{22} \quad \mbox{ and }\quad \overline{a}\mapsto e_{12}$$ defines an isomorphism of $(G,*)$-algebras in the following cases:
\begin{enumerate}
    \item[1.]  if  $s^\epsilon=h^+$\black, we have $R/I \cong C_2^h$.
    \item[2.] if $s^\epsilon=g^-$\black, we obtain $R/I \cong C_{2,*}^g$.  
\end{enumerate}

The converse in both cases follows from  Table \ref{cocharacterofc3} and Remark \ref{remark}. 

\end{proof}

\begin{lemma}\label{jordan} For $g,h \in G$ and $h\neq 1$ we have
    \begin{enumerate}
        \item $m_{((2)_{g^-})} \neq 0$ if and only if $C_{3,*}^g \in \varGstar(A)$.
        \item $m_{((2)_{h^+})} \neq 0$ if and only if $C_{3}^h \in \varGstar(A)$.
    \end{enumerate}
\end{lemma}

\begin{proof}

First, observe that if \( C_{3,*}^g \in \varGstar(A) \), then, by Remark~\ref{remark} and Table~\ref{cocharacterofc3}, we have \( m_{((2)_{g^-})} \neq 0 \). Similarly, if \( C_{3}^h \in \varGstar(A) \), then \( m_{((2)_{h^+})} \neq 0 \).

Now we assume that \( m_{((2)_{s^\epsilon})} \neq 0 \) for some \( s^\epsilon \in \{g^-, h^+\} \) with $g,h \in G$ and $h\neq 1$. By Proposition~\ref{prop:nonzero_multiplicities}, it follows that \( (x_{1,s}^\epsilon)^2 \notin \Idd(A) \). Hence, there exists a nonzero element \( a \in J(A)_s^\epsilon \) such that \( a^2 \neq 0 \). Let \( R \) be the \((G,*)\)-subalgebra of \( A \) generated by \( 1_F \) and \( a \), and let \( I \) the \((G,*)\)-ideal generated by \( a^3 \). The map $\varphi: R/I\rightarrow C_3$ given by 
$$\overline{1}\mapsto e_{11}+e_{22}+e_{33} ,\quad \overline{a}\mapsto e_{12}+e_{23} \quad \mbox{ and }\quad \overline{a^2}\mapsto e_{13}$$
defines an isomorphism of $(G,*)$-algebras in the following cases:
\begin{enumerate}
    \item[1.] if  $s^\epsilon=g^-$ \black then \( R/I \cong C_{3,*}^g \);
    \item[2.] if  $s^\epsilon=h^+$ \black then \( R/I \cong C_{3}^h \). 
\end{enumerate}
    Therefore, the proof follows. 
\end{proof}

\begin{lemma}\label{lemma1-1} For $g,h\in G$ and $h \neq 1$ we have
    \begin{enumerate}
        \item $m_{((1,1)_{1^+})}\neq 0$ if and only if $U_{3,*}^1 \in \varGstar(A)$.
        \item $m_{((1,1)_{g^-})}\neq 0$ if and only if $\mathcal{G}_{2,\tau}^{g,g} \in \varGstar(A)$.
        \item $m_{((1,1)_{h^+})} \neq 0$ if and only if $\mathcal{G}_{2,\psi }^{h,h}\in \varGstar(A)$.
    \end{enumerate}
\end{lemma}
\begin{proof}
    Assume that $m_{((1,1)_{1^\epsilon})}\neq 0$ with $\epsilon\in \{+,-\}$. By Proposition \ref{prop:nonzero_multiplicities} we have $[x_{1,1}^\epsilon,x_{2,1}^\epsilon]\neq 0$ and then,  by \cite[Lemmas 4.3 and 4.4]{Wesley}, it follows that
    
    \begin{enumerate}
        \item[1.] if $\epsilon=+$ then $U_{3,*}^1\in \varGstar(A)$;
        \item[2.] if $\epsilon=-$ then $\mathcal{G}_{2,\tau}^{1,1} \in \varGstar(A)$.
    \end{enumerate}

    Now assume $m_{((1,1)_{s^\epsilon})} \neq 0$ for $s \in G^\times$ and $\epsilon \in \{+,-\}$.  By Proposition \ref{prop:nonzero_multiplicities}, we have $[x_{1,s}^\epsilon,x_{2,s}^\epsilon ]\notin \textnormal{Id}^\sharp (A)$.

        Suppose that $|s|=2$. Hence, $B=F+J(A)_s$ is a $\gi$-subalgebra of $A$ with induced $\mathbb{Z}_2$-grading. Then, by \cite[Lemma 8.4]{Nascimento}, we have the following

\begin{enumerate}
    \item[1.] if $\epsilon=-$ we have $\mathcal{G}_{2,\tau}^{s,s}\in \varGstar(A)$;
    \item[2.] if $\epsilon=+$ we have $\mathcal{G}_{2,\psi}^{s,s}\in \varGstar(A)$.
\end{enumerate}

   Therefore, we now assume that $|s|>2$. 
   
   Suppose that $\epsilon=-$. Since $A$ is a unitary algebra with quadratic codimension growth, it follows that
    $$ \Gamma_{n}^\sharp=\Gamma_{n}^\sharp\cap \Idd(A) = \Gamma_n^{\sharp} \cap \Idd(\mathcal{G}_{2,\tau}^{s,s}), \quad \textnormal{for all} \quad n\geq 3. $$
    Furthermore, since $[x_{1,s}^-,x_{2,s}^-]\notin \textnormal{Id}^\sharp (A)$ we have $\{1^+,s^-, {(s^2)}^-\}\subset \text{supp}^*(A)$. As $\text{supp}^*(\mathcal{G}_{2,\tau}^{s,s})=\{1^+,s^-, {(s^2)}^-\}$ then
    $$ \Gamma_1^{\sharp} \cap \Idd(A) \subseteq \Gamma_1^{\sharp}\cap \Idd(\mathcal{G}_{2,\tau}^{s,s}). $$
    We now analyze the proper multilinear identities of degree 2 of $A$.  Let $f\in \Gamma_2^{\sharp}\cap \Idd(A)$ a multihomogeneous identity. By Lemma \ref{T-ideal_G_2tau} the polynomials
    $$[x_{1,1}^+,y],\quad  x_{1,s}^-x_{2,s^2}^-, \quad x_{1,s^2}^-x_{2,s^2}^-,\quad  (x_{1,s}^-)^2,  \quad x_{1,1}^-,\quad  x_{1,s}^+, \quad x_{1,s^2}^+, 
    \quad x_{1,h}  $$
are identities of $\mathcal{G}_{2,\tau}^{s,s}$, for all
$y\in \{x_{2,1}^+,x_{2,s}^-,x_{2,s^2}^-\}$ and $h\in G-\{1,s,s^2\}$. Hence, either $f\in \Idd(\mathcal{G}_{2,\tau}^{s,s})$ or $f=\alpha x_{1,s}^-x_{2,s}^-+ \beta [x_{1,s}^-,x_{2,s}^-]\in \Gamma_{(2_{s^-})}$ with $\alpha,\beta \neq 0.$ In the second case, \[
    \gamma x_{1,s}^-x_{2,s}^- - x_{2,s}^-x_{1,s}^-\equiv 0 \mbox{ on }A, \quad \text{where } \gamma = \frac{\alpha+\beta}{\beta}.
\]
Applying the involution on the polynomial above, we obtain  $\gamma x_{2,s}^-x_{1,s}^- -x_{1,s}^-x_{2,s}^-\equiv 0$ on $A$.  Combining both relations, we obtain $  (\gamma^2-1) x_{1,s}^-x_{2,s}^- \equiv 0$  on $A$. Since $x_{1,s}^-x_{2,s}^- \not\equiv 0$ on $A$, then $\gamma^2=1$.

Since $[x_{1,s}^-,x_{2,s}^-]$ is not an identity of $A$, then $\gamma=-1$ and $f=-\beta x_{1,s}^-\circ x_{2,s}^- \in \Idd(\mathcal{G}_{2,\tau}^{s,s})$. Thus,   $$ \Gamma_2^{\sharp}\cap\Idd(A) \subseteq \Gamma_2^{\sharp} \cap \Idd(\mathcal{G}_{2,\tau}^{s,s}). $$
    and so $\mathcal{G}_{2,\tau}^{s,s}\in \varGstar(A)$. 
    
    Similarly, if $\epsilon=+$ we can use Lemma \ref{T-ideal_G_2gamma} and the previous argument to prove that $\mathcal{G}_{2,\psi}^{s,s}\in \varGstar(A)$. 
    
    The converse follows from  Remark~\ref{remark} and Tables \ref{table_N3_U3}, \ref{tableG_2tau},  \ref{tableG_2gamma_psi}.
\end{proof}

\begin{lemma}\label{Lemma1+} For $g ,h\in G$ with $h\neq 1$ we have 
    \begin{enumerate}
        \item $m_{((1)_{1^+}, (1)_{h^+})}\neq 0$ if and only if $U_{3,*}^h \in \varGstar(A)$.
        \item $m_{((1)_{1^+}, (1)_{g^-})}\neq 0$ if and only if $N_{3,*}^g \in \varGstar(A)$.
    \end{enumerate}
\end{lemma}
\begin{proof}

  First of all, we notice that, if $m_{((1)_{1^+}, (1)_{1^-})}\neq 0$, by \cite[Lemmas 4.3]{Wesley}, we have $N_{3,*}^1\in \textnormal{var}^\sharp(A)$.  So we assume $m_{((1)_{1^+}, (1)_{h^\epsilon})}\neq 0$ where $h\neq 1$ and $\epsilon \in \{+, -\}$. By Proposition  \ref{prop:nonzero_multiplicities} we have $[x_{1,1}^+, x_{2,h}^\epsilon]\notin  \Idd(A)$. 
    
    Assume that $\epsilon=+$. Since $U_{3,*}^h$ and $A$ are unitary $(G,*)$-algebras with quadratic codimension growth then we have
    $$ \Gamma_{n}^\sharp\cap \Idd(A) = \Gamma_n^{\sharp} \cap \Idd(U_{3,*}^h), \quad \textnormal{for all} \quad n\geq 3. $$
    Furthermore, since $m_{((1)_{1^+},(1)_{h^+})} \neq 0$ on $A$, we have
         $\{1^+,h^+,h^-\}\subset \text{supp}^*(A)$. As  $\text{supp}^*(U_{3,*}^h)=\{1^+,h^+,h^-\}$  it follows that
    $$ \Gamma_1^{\sharp} \cap \Idd(A) \subseteq \Gamma_1^{\sharp}\cap \Idd(U_{3,*}^h). $$
    
    
    Therefore, we now analyze the proper multilinear identities of degree $2$ of $A$. Let $f \in \Gamma_2^\sharp \cap\Idd(A)$ a multihomogeneous identity. By Lemma \ref{T-ideal_N_e_U} the polynomials $$x_{1,1}^-, \quad [x_{1,1}^+,x_{2,1}^-], \quad x_{1,h}^+x_{2,h}^+, \quad x_{1,h}^-x_{2,h}^-, \quad x_{1,h}^+x_{2,h}^-, \quad x_{1,r}^\sigma$$ are identities of $U_{3,*}^h$, where $r\in G-\{1,h\}$. Hence, it follows that either $f\in \Idd(U_{3,*}^h)$ or $f \in \Gamma_{(1_{1^+}, 1_{h^+})}$. Since $[x_{1,1}^+,x_{1,h}^+]$ is not an identity of $U_{3,*}^h$, then we have $f\in \Idd(U_{3,*}^h)$. Therefore, 
$\Gamma_2^{\sharp}\cap\Idd(A) \subseteq \Gamma_2^{\sharp} \cap \Idd(U_{3,*}^h) $ and so $U_{3,*}^h\in \varGstar(A)$. 
    
   In a similar way, in case $m_{((1)_{1^+},(1)_{h^-})}\neq 0$, we can use Lemma \ref{T-ideal_N_e_U} and the previous argument to conclude that $N_{3,*}^h\in \varGstar(A)$.

Finally, the converse of items (1) and (2) follows by  Remark \ref{remark} and Table \ref{table_N3_U3}. 
\end{proof}

\begin{lemma}\label{ultima_mult_1} For $g,h\in G^\times$, $u\in G$, with $g\neq h$, $h \neq u$, we have 
\begin{enumerate}
    \item  $m_{((1)_{g^+},(1)_{h^+})}= 1$ if and only if $\mathcal{G}_{2,\psi}^{g,h}\in \varGstar(A)$ or $W_{\nu_1}^{g,h} \in \varGstar(A)$.
    
    \item  $m_{((1)_{u^-},(1)_{h^-})}=1$ if and only if $\mathcal{G}_{2,\tau}^{u,h}\in \varGstar(A)$ or $W_{\nu_2}^{u,h} \in \varGstar(A)$.
    
    \item  $m_{((1)_{u^-},(1)_{g^+})}=1$ if and only if $\mathcal{G}_{2,\gamma}^{u,g}\in \varGstar(A)$ or $W_{\nu_3}^{u,g} \in \varGstar(A)$.
\end{enumerate}
\end{lemma}

\begin{proof}

    Assume $m_{((1)_{s^\epsilon},(1)_{t^\delta})}= 1$, where $s^\epsilon \in \{g^+,u^-\}$, $t^\delta\in \{h^+,h^-,g^+\}$, with $(s^\epsilon, t^\delta)= (u^- , g^+)$ or $s\neq t$ and $\epsilon = \delta$. Since $x_{1,s}^\epsilon x_{2,t}^\delta$ and $[x_{1,s}^\epsilon ,x_{2,t}^\delta]$ are proper h.w.v.'s in $\Gamma_{((1)_{s^\epsilon},(1)_{t^\delta})}$, by Proposition \ref{prop:nonzero_multiplicities} there exists $(\alpha,\beta)\neq (0,0)$ such that $\alpha x_{1,s}^\epsilon x_{2,t}^\delta +\beta [x_{1,s}^\epsilon ,x_{2,t}^\delta] \equiv 0$ on $A$. Moreover, since the involution is graded, we can consider only the  cases 
    \begin{enumerate}
        \item[1.]  $\alpha= 0$ and $\beta\neq 0$;
        \item[2.]  $\alpha \neq 0$ and $\beta\neq0$.
    \end{enumerate}
    
    In the first case, note that $x_{1,s}^\epsilon x_{2,t}^\delta \not\equiv  0$ and $[x_{1,s}^\epsilon ,x_{2,t}^\delta]\equiv 0$, hence there exist $a\in J(A)_{s}^\epsilon$ and $b\in J(A)_{t}^\delta$ such that $ab \neq 0$ and $ab=ba$. Let $R$ be the $\gi$-subalgebra of $A$ generated by $1_F, a, b$ and $I$ the $\gi$-ideal generated by $a^2$ and $b^2$. Since $A$ has quadratic codimension growth, we have $\Gamma_{n}^{\sharp} \subseteq \IdGstar{A}$, for all $n\geq 3$ and so $I$ is linearly generated by $a^2$ and $b^2$. 
    Moreover, since $\supp(I)=\{t^2,s^2\}$ , we have $a,b, ab \notin I$. Now, we observe that the map $\varphi: R/I \rightarrow W$ given by
    $$\overline{1_F} \mapsto e_{11}+\cdots +e_{44}, \quad \overline{a} \mapsto e_{12}+e_{34}, \quad \overline{b}\mapsto e_{13}+e_{24}, \quad \overline{ab} \mapsto e_{14}$$ defines an isomorphism of $\gi$-algebras in the following cases
    \begin{enumerate}
        \item[1.] If $s^\epsilon=g^+$ and $t^\delta=h^+$ then $R/I \cong W_{\nu_1}^{g,h}$;
        \item[2.] If $s^\epsilon=u^-$ and $t^\delta=h^-$ then $R/I \cong W_{\nu_2}^{u,h}$;
        \item[3.] If $s^\epsilon=u^-$ and $t^\delta=g^+$ then $R/I \cong W_{\nu_3}^{u,g}$.
    \end{enumerate}  

    In the second case, note that $\gamma x_{1,s}^\epsilon x_{2,t}^\delta\equiv x_{2,t}^\delta x_{1,s}^\epsilon$ on $A$, where $\gamma=\frac{\alpha+\beta}{\beta}$. Now applying the involution on both sides of the equation, we obtain 
  $$(\gamma^2-1) x_{1,s}^\epsilon x_{2,t}^\delta\equiv 0\mbox{ on }A$$ and therefore   $\gamma^2=1$. Since $\alpha \neq 0$, we must have $\gamma=-1$ and then $  x_{1,s}^\epsilon x_{2,t}^\delta+ x_{2,t}^\delta x_{1,s}^\epsilon\equiv 0$ on $A$.

    Now, we consider $a\in J(A)_s^\epsilon$ and $b\in J(A)_{t}^\delta$ such that $ab\neq 0$. By the previous identity, we have $ab= -ba$. Consider $R$ the $\gi$-subalgebra of $A$ generated by $1_F,a,b$ and $I$ the $\gi$-ideal of $A$ generated by $a^2$ and $b^2$. It is easy to check that the map $\psi: R/I\rightarrow \mathcal{G}_2$ given by
    $$\overline{1_F} \mapsto 1, \quad \overline{a}\mapsto e_1, \quad \overline{b}\mapsto e_2, \quad \overline{ab} \mapsto e_1e_2$$ is an isomorphism of $(G,*)$-algebras in the following cases: 
    \begin{itemize}
        \item[1.]if $s^\epsilon=g^+$ and $t^\delta=h^+$ then $R/I \cong \mathcal{G}_{2,\psi}^{g,h}$;
        \item[2.]if $s^\epsilon=u^-$ and $t^\delta=h^-$ then $R/I \cong \mathcal{G}_{2,\tau}^{u,h}$;
        \item[3.]if $s^\epsilon=u^-$ and $t^\delta=g^+$  then $R/I \cong \mathcal{G}_{2,\gamma}^{u,g}$.
    \end{itemize}

    For the converse, it is enough to use Remark \ref{remark} and the respectively proper cocharacters in Tables \ref{tableG_2tau}, \ref{tableG_2gamma_psi} and \ref{table_W}.
\end{proof}
    
Now, we analyze the cases where the multiplicities in the previous lemma are equal to $2$.

\begin{lemma}\label{mult_2} For $g,h\in G^\times$, $u \in G$ with $g\neq h$ and $h \neq u$, we have 
\begin{enumerate}
    \item  $m_{((1)_{g^+},(1)_{h^+})}= 2$ if and only if $\mathcal{G}_{2,\psi}^{g,h}\oplus W_{\nu_1}^{g,h} \in \varGstar(A)$.
    \item  $m_{((1)_{u^-},(1)_{h^-})}=2$ if and only if $\mathcal{G}_{2,\tau}^{u,h}\oplus W_{\nu_2}^{u,h} \in \varGstar(A)$.
    \item  $m_{((1)_{u^-},(1)_{g^+})}=2$ if and only if $\mathcal{G}_{2,\gamma}^{u,g}\oplus W_{\nu_3}^{u,g} \in \varGstar(A)$.
\end{enumerate}
\end{lemma}
 
\begin{proof}
Assume $m_{((1)_{s^\epsilon},(1)_{t^\delta})}= 2$,  for some $s^\epsilon \in \{g^+,u^-\}$, $t^\delta\in \{h^+,h^-,g^+\}$, with $(s^\epsilon, t^\delta)= (u^- , g^+)$ or $s\neq t$ and $\epsilon = \delta$. Let $B$ be the corresponding $(G,*)$-algebra associated to the multiplicity $m_{((1)_{s^\epsilon},(1)_{t^\delta})}$ described in the statement $(1),(2)$ and $(3)$ above. Since $B$ has quadratic codimension growth,
    $$\Gamma_{n}^\sharp= \Gamma_n^{\sharp} \cap \IdGstar{A} =\Gamma_n^{\sharp} \cap \IdGstar{B},\quad \text{for all} \quad n \geq 3.$$

  Observe that $\supp^*(B)=\{1^+,s^\epsilon,t^\delta,(st)^+,(st)^-\}$. Since $A$ is a unitary algebra with $m_{((1)_{s^\epsilon},(1)_{t^\epsilon})}\neq 0$ then by Proposition \ref{prop:nonzero_multiplicities} we have $[x_{1,s}^\epsilon,x_{1,t}^\delta]\notin \IdGstar{A}$. Thus, it follows that  $\supp^*(B)\subseteq \supp ^*(A)$ and $$\Gamma_1^{\sharp}\cap\IdGstar{A}\subseteq\Gamma_1^{\sharp}\cap\IdGstar{B}.$$

    Now, we consider $f$ a multilinear polynomial in $\Gamma_2^{\sharp}\cap\Idd(A)$. Without loss of generality, we may assume that $f$ is multihomogeneous. According to Lemma \ref{Direct_sum}, we have the following cases
    \begin{itemize}
        \item[1.] If $s^\epsilon=g^+$ and $ t^\delta=h^+$ then either $f \in \Idd(\mathcal{G}_{2,\psi}^{g,h}\oplus W_{\nu_1}^{g,h})$ or $f=\alpha x_{1,g}^+ x_{2,h}^++\beta[x_{1,g}^+ ,x_{2,h}^+]$;
        \item[2.] If $s^\epsilon=u^-$ and $ t^\delta=h^-$ then either $f \in \Idd(\mathcal{G}_{2,\tau}^{u,h}\oplus W_{\nu_2}^{u,h})$ or $f=\alpha x_{1,u}^- x_{2,h}^-+\beta[x_{1,u}^- ,x_{2,h}^-]$;
        \item[3.] If $s^\epsilon=u^-$ and $ t^\delta=g^+$ then either $f \in \Idd(\mathcal{G}_{2,\gamma}^{u,g}\oplus W_{\nu_3}^{u,g})$ or $f=\alpha x_{1,u}^- x_{2,g}^++\beta[x_{1,u}^- ,x_{2,g}^+]$.
    \end{itemize}

     Since we are assuming that $m_{((1)_{s^\epsilon},(1)_{t^\delta})}=2$,
Proposition~\ref{prop:nonzero_multiplicities} implies that no linear combination
of the proper highest weight vectors
$x_{1,s}^\epsilon x_{2,t}^\delta$ and $[x_{1,s}^\epsilon,x_{2,t}^\delta]$
can be an identity of $A$. Therefore, the second case in each item above
cannot occur.

    Thus, we have proved that $\Gamma_n^\sharp \cap \textnormal{Id}^\sharp (A)\subseteq \Gamma_n^\sharp\cap \textnormal{Id}^\sharp (B) $, for all $n$. Since $A$ is unitary, the proof follows. 
    
    The converse follows from Remark \ref{remark} and Tables \ref{table_sum1} and \ref{table_sum2}.
\end{proof}
 \black

For an algebra $A$, recall that the algebra $\widetilde{A} = A \times F$ endowed with the product
\[
(a_1,\alpha_1)(a_2,\alpha_2)
= \big(a_1a_2 + \alpha_2 a_1 + \alpha_1 a_2,\; \alpha_1\alpha_2\big),
\]
is the algebra obtained from $A$ by adjoining a unity. If $A$ is a $(G,*)$-algebra with involution ${*}$, then $\widetilde{A}$ becomes a $(G,*)$-algebra with $G$-grading given by setting $\widetilde{A}_1 = (A_1, F)$ and { $\widetilde{A}_g=(A_g,\{0\})$,} for all $g\neq1$, and involution $\diamond$ being $(a,\alpha)^{\diamond} = (a^{*},\alpha)$.

\begin{lemma} \label{btildeinvar}
    Let $A$ be a unitary $(G,*)$-algebra. If $B\in \textnormal{var}^
    \sharp(A)$ then $\widetilde{B}\in \textnormal{var}^\sharp(A)$.
\end{lemma}

\begin{proof}
   Let $f \in \textnormal{Id}^\sharp(A)$ be a multilinear $(G,*)$-polynomial. Since both $A$ and $\widetilde{B}$ are unitary $(G,*)$-algebras, we may assume that $f$ is a proper $(G,*)$-polynomial. Let $\lambda$ be an evaluation of $f$ on elements of $\widetilde{B}$. As $f$ is proper and multilinear, we observe that
\[
\lambda(f) = (\lambda_1(f), 0),
\]
for some evaluation $\lambda_1$ of $f$ on $A$. The conclusion now follows from the fact that $f$ is an identity of $A$.

\end{proof}

\begin{lemma} \label{unitariornilpotentorcummu}
    Let $A$ be a unitary finite-dimensional $(G,*)$-algebra of polynomial codimension growth without unity. Then either $
    \widetilde{A}$ has exponential growth or $A\sim_{T_{G}^*}N$ or $A\sim_{T_{G}^*} C\oplus N$, where $N$ is a nilpotent $(G,*)$-algebra and $C$ is a commutative algebra with trivial grading and trivial involution.  
\end{lemma}

\begin{proof}
    
    Since $A$ is a finite-dimensional $\gi$-algebra with polynomial codimension growth, by \cite[Theorem 4.2]{Maralice} we may assume that $$A=A_1\oplus\cdots \oplus A_k \dotplus J.$$
    where $J$ denotes the Jacobson radical of $A$, each $A_i\cong F$ with trivial involution, for all $i$, and $A_lJA_m =0$ for all $l\neq m$. 

Observe that, if $A=A_1\oplus \cdots \oplus A_k \oplus J$ as a direct summand, then we have $A\sim_{T_G^*}C\oplus N$ or $A\sim_{T_{G}^*}N$. 

    Otherwise, there must exist $1\leq i\leq k$ such that $A_iJ \neq 0$.  Since $\widetilde{A}=(A_1\oplus \cdots \oplus A_{k} \dotplus J) \times F$, we obtain
$$
\widetilde{A}=\bar{A}_{1}\oplus\cdots \oplus \bar{A}_{k}\oplus \bar{F}\dotplus\bar{J},
$$
where $\bar{A}_{i}=\{(a_{i},0)\mid a_{i}\in A_{i}\}$, { $\bar{F}=\{(0,\alpha)\mid \alpha\in F\}\cong F$} and $\bar{J}=\{(j,0)\mid j\in J\}$. Now, note that for $l\neq m$, we have $\bar{A}_{l}\bar{J}\bar{A}_{m}=0$. However, since $\bar{A}_{i}\bar{J}\, \overline{F}=\bar{A}_{i}\bar{J}\ne0$, by \cite[Theorem 4.2]{Maralice} we obtain $\widetilde{A}$ with exponential codimension growth.
    
\end{proof}

\begin{corollary} 
\label{unitarunilpotentcomm}
    Let $A$ be a unitary $(G,*)$-algebra of polynomial codimension growth. If $B\in \textnormal{var}(A)$ is a finite-dimensional algebra then either $B$ is unitary or $B\sim_{T_G^*}N$ or $B\sim_{T_G^*}C\oplus N$, where $C$ is a commutative $(G,*)$-algebra with trivial involution and trivial grading and $N$ is a nilpotent algebra.
\end{corollary}

\begin{proof}
Suppose that $B$ is a nonunitary $(G,*)$-algebra such that $B \in \textnormal{var}^\sharp(A)$. Since $A$ has polynomial growth, it follows that $B$ also has polynomial growth. Moreover, by Lemma~\ref{btildeinvar}, we have $\widetilde{B} \in \textnormal{var}^\sharp(A)$, and therefore $\widetilde{B}$ has polynomial growth. The proof now follows from Lemma~\ref{unitariornilpotentorcummu}.
\end{proof}

Before we present the main result of this section, we define the following set of $(G,*)$-algebras
$$ \mathcal{I}=\{ F, C_{2}^{ {h}}, C_{2,*}^{ {g}}, C_{3}^{ {h}},C_{3,*}^{ {g}}, U_{3,*}^{ {g}},N_{3,*}^{ {h}},\mathcal{G}_{2,\tau}^{g,g}, \mathcal{G}_{2,\psi }^{h,h},  \mathcal{G}_{2,\psi}^{p,q},  \mathcal{G}_{2,\tau}^{r,q},\mathcal{G}_{2,\gamma}^{r,q}, W_{\nu_1}^{p,q},W_{\nu_2}^{r,q},  W_{\nu_3}^{r,q} \} $$ 
    where $g,h,p, q, r \in G$, $h,p,q\neq 1, q\neq r$. 

Also, we consider
$$\widetilde{\mathcal{I}}=\mathcal{I}-\{ F,C_{2}^{ {h}},C_{2,*}^{ {g}} ~|~ g, h\in G ,h \neq 1\}.$$
\begin{theorem}
    Let $A$ be a finite-dimensional unitary $\gi$-algebra over a field $F$ of characteristic zero with quadratic codimension growth. Then, $A$ is T${_G^*}$-equivalent to a finite direct sum of $\gi$-algebras in the set $\mathcal{I},$ where at least one algebra in the set $\widetilde{\mathcal{I}}$ appears as a component of such a direct sum. 
\end{theorem} 
\begin{proof}
Since $A$ has polynomial growth, by Theorem~\ref{estrutura}, it is $T_G^*$-equivalent to a finite direct sum
\[
A \sim_{T_G^*} A_1 \oplus \cdots \oplus A_m,
\] where each $A_i$ is a finite-dimensional $\gi$-algebra which is either nilpotent or of type $F + J(A_i)$.   
Since $A$ has quadratic codimension growth, it follows that
$A_1 \oplus \cdots \oplus A_m$ also has quadratic codimension growth.
Therefore, by Corollary~\ref{unitarunilpotentcomm}, the algebra
$A_1 \oplus \cdots \oplus A_m$ is unitary. Consequently,
\[
A_t = F + J(A_t), \quad \text{for all } t = 1, \ldots, m,
\]
and there exists an index $i$ such that $A_i$ has quadratic growth of the
sequence of $(G,*)$-codimensions. \black

In this case, 
according to Corollary \ref{unitarunilpotentcomm}, we may also assume that $A_i$  is unitary and so, we have $\Gamma_n^\sharp \subseteq \Idd(A_i)$ for all $n \geq 3$. For $n=1$ and $n=2$, we consider $n = n_1 + \cdots + n_{2k}$ and analyze the multiplicities in the proper $(n_1,\ldots,n_{2k})$-cocharacters of $A_i$. 
Recall that, by~(\ref{multiplicities}), we have
$$   0\leq m_{\lambda }\leq 1 \quad \mbox{ and }\quad\quad 0\leq m_{((1)_{s^\delta},(1)_{t^\epsilon})}\leq 2,$$ for all $\lambda \in \mathcal{S}=\{((1)_{u^{\epsilon_{1}}}),((2)_{u^{\epsilon_1}}),(1,1)_{u^{\epsilon_2}},((1)_{1^+},(1)_{u^{\epsilon_1}})\mid u\in G, \epsilon_i\in \{+,-\},  u^{\epsilon_1}
\neq 1^+\}$ and $s^\delta\neq t^\epsilon \neq 1^+$.

\smallskip

First we consider $n = 1$. In this case, if $m_{((1)_{u^{\epsilon_1}})} = 1$ then by Lemma~\ref{lineares1} either $u^{\epsilon_1} = h^+$ and $C_2^h \in \varGstar(A_i)$, or $u^{\epsilon_1} = g^-$ and $C_{2,*}^{g} \in \varGstar(A_i)$, for some $g,h \in G$ with $g \neq 1$.

\smallskip

Now consider the case $n = 2$. 
By Lemma~\ref{jordan}, if $m_{((2)_{u^{\epsilon_1}})} = 1$ then either $u^{\epsilon_1} = g^+$ and $C_3^g \in \varGstar(A_i)$, or $u^{\epsilon_1} = h^-$ and $C_{3,*}^h \in \varGstar(A_i)$, for some $g,h \in G$ with $g \neq 1$.

\smallskip

According to Lemma~\ref{lemma1-1}, if $m_{((1,1)_{u^{\epsilon_2}})} = 1$ then either $u^{\epsilon_2} = 1^+$ and $U_{3,*}^1 \in \varGstar(A_i)$, or $u^{\epsilon_2} = g^-$ and $\mathcal{G}_{2,\tau}^{g,g} \in \varGstar(A_i)$, or $u^{\epsilon_2} = h^+$ and $\mathcal{G}_{2,\psi}^{h,h} \in \varGstar(A_i)$, for some $g,h \in G$ with $h \neq 1$.

\smallskip

If $m_{((1)_{1^+},(1)_{u^{\epsilon_1}})} = 1$ then by Lemma~\ref{Lemma1+} either $u^{\epsilon_1} = h^+$ and $U_{3,*}^h \in \textnormal{var}^\sharp(A_i)$, or $u^{\epsilon_1} = g^-$ and $N_{3,*}^g \in \varGstar(A_i)$, for some $g,h \in G$ with $h \neq 1$.

\smallskip

Finally, if $m_{((1)_{s^\epsilon},(1)_{u^\sigma})} \neq 0$, we must consider the following cases. 
If $m_{((1)_{s^\delta},(1)_{t^\epsilon})} = 1$, then by Lemma~\ref{ultima_mult_1} either
\begin{itemize}
    \item $(s^\delta,t^\epsilon) = (g^+,h^+)$ and $\mathcal{G}_{2,\psi}^{g,h}$ or $W_{\nu_1}^{g,h}$ belong to $\varGstar(A_i)$;
    \item $(s^\delta,t^\epsilon) = (g^-,h^-)$ and $\mathcal{G}_{2,\tau}^{g,h}$ or $W_{\nu_2}^{g,h}$ belong to $\varGstar(A_i)$;
    \item $(s^\delta,t^\epsilon) = (g^-,h^+)$ and $\mathcal{G}_{2,\gamma}^{g,h}$ or $W_{\nu_3}^{g,h}$ belong to $\varGstar(A_i)$.
\end{itemize}

The case $m_{((1)_{s^\epsilon},(1)_{t^\delta})} = 2$ is given by Lemma~\ref{mult_2}. 
In this situation, when $(s^\epsilon,t^\delta) = (g^+,h^+)$ we have $\mathcal{G}_{2,\psi}^{g,h} \oplus W_{\nu_1}^{g,h} \in \varGstar(A_i)$; 
if $(s^\epsilon,t^\delta) = (g^-,h^-)$ then $\mathcal{G}_{2,\tau}^{g,h} \oplus W_{\nu_2}^{g,h} \in \varGstar(A_i)$; 
and when $(s^\epsilon,t^\delta) = (g^-,h^+)$ we have $\mathcal{G}_{2,\gamma}^{g,h} \oplus W_{\nu_3}^{g,h} \in \varGstar(A_i)$.

\smallskip

Moreover, we recall that at least one of the multiplicities
\[
m_{((2)_{u^{\epsilon_1}})}, \quad 
m_{((1,1)_{u^{\epsilon_2}})}, \quad  
m_{((1)_{1^+},(1)_{u^{\epsilon_1}})}, \quad 
m_{((1)_{s^\delta},(1)_{t^\epsilon})}
\]
must be nonzero.

\smallskip

Observe that, for each nonzero multiplicity, there exists a corresponding $(G,*)$-algebra belonging to the variety generated by $A_i$. 
Let $B$ denote the $\gi$-algebra obtained as the direct sum of the $\gi$-algebras associated with the nonzero multiplicities appearing in the decomposition of the proper $(n_1,\ldots,n_{2k})$-cocharacters of $A_i$, for $n = n_1 + \cdots + n_{2k}$ with $n \in \{1,2\}$. 
By Lemmas~\ref{lineares1}--\ref{ultima_mult_1}, we have $B \in \varGstar(A_i)$. 
Furthermore, by Remark~\ref{remark}, the algebras $A_i$ and $B$ have the same multiplicities in the decomposition of all proper $(n_1,\ldots,n_{2k})$-cocharacters. 
Therefore, $c_n^\sharp(A_i) = c_n^\sharp(B)$, and consequently $\varGstar(B) = \varGstar(A_i)$.

\smallskip

Finally, observe that if $A_t$ has at most linear codimension growth for some $1 \leq t \leq m$, then by~\cite[Theorem~6.5]{Maralice}, $A_t$ is $T_G^*$-equivalent to a finite direct sum of $\gi$-algebras belonging to the set  $$\{F,C_{2}^{ {h}},C_{2,*}^{ {g}} ~|~ g, h\in G ,h \neq 1\}.$$

Finally, recalling that $A \sim_{T_G^*} A_1 \oplus \cdots \oplus A_m$ and at least one of the components $A_i$ has quadratic codimension growth, the result then follows. 

\end{proof}

In \cite{Maralice}, the authors provided a complete classification of the varieties generated by finite-dimensional $\gi$-algebras with at most linear codimension growth. In particular, they proved that any such variety is generated by a finite direct sum of finite-dimensional $\gi$-algebras generating minimal varieties, each one having at most linear codimension growth. Combining this classification with the preceding theorem, we derive the following result.
    
\begin{corollary}
    Let $A$ be a finite-dimensional unitary $\gi$-algebra. Then $c_n ^\sharp(A)\leq \alpha n^2$ if and only if $A$ is $T_G^*$-equivalent to a finite direct sum of algebras generating minimal varieties with at most quadratic codimension growth.
\end{corollary}

\textbf{Funding} The first author was partially supported by FAPESP, grant no. 2025/5699-0. The second author  was partially supported by CAPES. The third author was partially supported by CNPq.

\textbf{Data Availability} No relevant declarations are applicable. \vspace{0.2cm}

\textbf{ \large Declarations} \vspace{0.2cm}

\textbf{Ethical Approval} Not applicable.

\textbf{Competing interests} The authors declare no competing interests.



\begin{thebibliography}{999}
\bibitem{Malu} D. C. L. Bessades and W. D. S. Costa and M. L. O. Santos. \newblock{\em On unitary algebras with graded involution of quadratic growth.} Linear Algebra Appl.  689 (2024) 260-293.

\bibitem{Dafne} D. Bessades, M. L. O. Santos, R. B. dos Santos and A. C. Vieira. \newblock{\em Superalgebras and algebras with involution: classifying varieties of quadratic growth.} Comm. Algebra 49 (2021) 2476-2490.

\bibitem{Wesley} W. Q. Cota. \newblock{\em Algebras with involution and $*$-colength bounded by 5.} Internat. J. Algebra Comput. 35 (2025) 1159-1180.

\bibitem{Wesley20} W. Q. Cota. \newblock{\em Group graded algebras and varieties with quadratic codimension growth}. Submitted.

\bibitem{Cota2} W. Q. Cota.  {\em Minimal $G$-graded varieties of quadratic growth and small colength.} J. Algebra Appl., DOI: 10.1142/S0219498826502087.

\bibitem{Wesley2} W. Q. Cota, R. B. dos Santos, A. C. Vieira. \newblock{\em On the colength sequence of algebras with graded involution}. Submitted.

\bibitem{Wesley3} W. Q. Cota and A. C. Vieira. \newblock{\em Minimal varieties of algebras with graded involution and quadratic growth}. Submitted.

{\bibitem{Willer} W. Costa, A. Ioppolo, R. dos Santos, A. C. Vieira, \newblock{\em Unitary superalgebras with graded involution or superinvolution of polynomial growth.} \newblock J. Pure Appl. Algebra {{225}} (2021) 106666.}


\bibitem{Drensky} V. Drensky.
		\newblock{ Free algebras and PI-algebras.} \newblock Singapore: Springer-Verlag. 2000.

  \bibitem{GiamLaMa}   A. Giambruno and D. La Mattina. \newblock{\em PI-algebras with slow codimension growth.} J. Algebra 284 (2005) 371-391.

  \bibitem{Misso} A. Giambruno, D. La Mattina and P. Misso.\textit{ Polynomial identities on superalgebras: classifying linear growth.} J. Pure Appl. Algebra 207 (2006) 215-240.


\bibitem{Petro} A. Giambruno, D. La Mattina and V. Petrogradsky. {\em Matrix algebras of polynomial codimension growth}. Israel J. Math 158 (2007) 367-378.

\bibitem{Tatiana} T. A. Gouveia, R. B. dos Santos and A. C. Vieira. {\em Minimal $*$-varieties and minimal supervarieties of polynomial growth}. J. Algebra 552 (2020) 107-133. 

  \bibitem{Sandra}   S. M. A. Jorge and A. C. Vieira. {\em  On minimal varieties of quadratic growth}. Linear Algebra Appl. 419 (2006) 925-938.

   \bibitem{Kemer} A. R. Kemer.{ Ideals of Identities of Associative Algebras}. AMS Translations of Math. Monographs. Vol 87, 1988.

        \bibitem{Kem} A. R. Kemer.  {\em Varieties of finite rank}. Proc. 15th All the Union Algebraic Conf., Krasnoyarsh.  2 (1979) (in Russian).

        \bibitem{Plamen} P. Koshlukov and D. La Mattina. \textit{Graded algebras with polynomial growth of their codimensions}. J. Algebra 434 (2015) 115-137.

        \bibitem{LaMa} D. La Mattina.  {\em Varieties of almost polynomial growth: classifying their subvarieties}. Manuscripta Math. 123 (2007) 185-203.



\bibitem{RN} T. S. do Nascimento, R. B. dos Santos and A. C. Vieira. \newblock{\em Graded cocharacters of minimal subvarieties of supervarieties of almost polynomial growth.} \newblock J. Pure Appl. Algebra { 219} (2015) 913-929.

\bibitem{Nascimento} T. S. do Nascimento and A. C. Vieira. \newblock{\em Superalgebras with graded involution and star-graded colength bounded by 3.} Linear Multilinear Algebra 67 (2019) 1999-2020.


\bibitem{Maralice} M. A. de Oliveira, R. B. dos Santos and A. C. Vieira. \newblock{\em Polynomial growth of the codimensions sequence of algebras with group graded involution.} Israel J. Math. 1 (2023) 1-27.

  \bibitem{Mara3} M. A. de Oliveira and A. C. Vieira. {\em Varieties of unitary algebras with small growth of codimensions}. Internat. J. Algebra Comput. 31 (2021) 257-277.



\bibitem{Lorena} L. M. Oliveira, R. B. dos Santos and A. C. Vieira. \newblock{\em Varieties of group graded algebras with graded involution of almost polynomial growth.} \newblock Algebr. Represent. Theory { 26} (2023) 663-677.

  \bibitem{RG}A. Regev. 
\newblock{\em Existence of identities in $A\otimes B$}. Israel J. Math. { 11} (1972) 131-152.

\bibitem{Sagan} B. E. Sagan.
		\newblock{ The symmetric group - representations, combinatorial algorithms and symmetric
			functions.}
		\newblock 1st ed. Belmont (CA): Wadsworth; 1991.


\end{thebibliography}
\end{document}